\documentclass[11pt]{article}
\usepackage{amssymb,amsmath,latexsym}
\usepackage{amsthm}
\usepackage[shortlabels]{enumitem}
\usepackage{caption}
\usepackage{xcolor}
\usepackage{float}
\usepackage{graphicx}
\usepackage{arydshln}
\usepackage[utf8]{inputenc}
\usepackage[english]{babel}
\usepackage[T2A]{fontenc}
\usepackage{tikz}

\usepackage{tabularx} 
\usepackage{colortbl} 
\usepackage{url}
\usepackage{breakurl}
\usepackage[colorlinks=true,citecolor=blue,linkcolor=blue,urlcolor=blue,bookmarks,bookmarksopen,bookmarksdepth=2,backref=page,breaklinks]{hyperref}

\usepackage{bookmark}
\bookmarksetup{
	numbered, 
	open,
}

\renewcommand*{\backref}[1]{}
\renewcommand*{\backrefalt}[4]{%
	\ifcase #1 (Not cited.)%
	\or        (Cited on page~#2.)%
	\else      (Cited on pages~#2.)%
	\fi}

\theoremstyle{plain}
\newtheorem{theorem}{Theorem}[section]
\newtheorem{lemma}[theorem]{Lemma}
\newtheorem{proposition}[theorem]{Proposition}

\newtheorem{corollary}[theorem]{Corollary}

\theoremstyle{definition}

\newtheorem{example}[theorem]{Example}
\newtheorem{openproblem}{Open Problem}
\newtheorem{definition}[theorem]{Definition}
\newtheorem{remark}[theorem]{Remark}

\newtheorem{notation}[theorem]{Notation}

\numberwithin{theorem}{subsection}
\numberwithin{equation}{subsection}
\numberwithin{table}{subsection}
\numberwithin{figure}{subsection}
\numberwithin{openproblem}{section}

\begin{document}
\title{Lattice walks, pattern statistics, and Riordan arrays}

\author{Milan Tenn\vspace{0.4cm} \ \\
Swarthmore College \\\tt mtenn1@swarthmore.edu\vspace{0.4cm}
}

\date{\today}
\maketitle
\abstract{We introduce a new class of Riordan matrices corresponding to lattice walks with colored steps. We show that this class corresponds to a class of pseudo-involutions and then find both generating functions and explicit combinatorial formulas for entries of these Riordan matrices. We use these Riordan matrices to find an explicit combinatorial formula for the occurences of an arbitrary pattern in our lattice walks. Finally, we define a distributive lattice on our lattice walks and connect our new RNA matrices to Whitney numbers of lattices.}\ \\[2mm]


\thispagestyle{empty}
\section{Introduction}
Peakless Motzkin paths and their prefixes are a class of lattice walks which correspond to a variety of applications. The enumeration of prefixes of peakless and valleyless Motzkin paths of length $n$ and endpoint height $k$ corresponds to entries of Riordan matrices respectively called RNA Array I and RNA Array II. Riordan matrices are a well-studied class of infinite lower triangular matrices which have been extensively investigated with both algebraic and combinatorial methods in~\cite{alternativeriordan, algebraicriordanstructure, riordansums, hepaper,deltasequence,riordansubgroupsbarry} and which are often useful in combinatorial counting problems related to lattice walks. RNA Arrays I and II in particular are called RNA Arrays because, in addition to counting lattice walks, their entries also count classes of RNA secondary structures with $n$ nucleotides and $k$ mutations. Nkwanta introduced RNA Arrays I and II and connected RNA secondary structures to lattice walks in~\cite{asamoahearlypaper}. As a result, combinatorial problems related to peakless and valleyless Motzkin paths often correspond to problems in RNA modeling. Furthermore, some generalizations of these RNA Arrays have been investigated in~\cite{evanspaper} by allowing $j$ colors for horizontal steps in peakless Motzkin paths and their prefixes. 

In this paper, we further generalize RNA Arrays I and II by introducing new Riordan arrays counting peakless and valleyless Motzkin paths while allowing $j_1$ colors for upwards steps, $j_2$ for horizontal steps, and $j_3$ for downwards steps. We derive both generating functions and explicit combinatorial formulas for these new generalized RNA Arrays. In addition, we show that the generalizations of RNA Array I are pseudo-involutions in the Riordan group whenever $j_1=1$, and we generalize certain relationships between RNA matrices found in~\cite{evanspaper}.

Our generalization of the RNA Arrays is in part motivated by potential applications to RNA secondary structures. From~\cite{asamoahearlypaper} it is well-known that the RNA Arrays correspond to RNA secondary structures. Adding colors might allow for the inclusion of additional information, such as information related to the RNA primary structure. For instance, we might 4-color downwards steps in peakless Motzkin paths, which correspond to nucleotides at the end of a bond, to encode which nucleotide actually appears in the primary structure at that location.

Furthermore, we consider a generalization of grand peakless Motzkin paths by allowing colors in each step. Grand peakless Motzkin paths relax the requirement that peakless Motzkin paths must stay at $y\ge 0$. They have previously been investigated in~\cite{grandgraycodes, grandtrees}. We find that the enumeration of grand peakless Motzkin paths and their prefixes of length $n$ and endpoint height $k$ yields a new RNA Array which we call RNA Array III. As with RNA Arrays I and II, allowing for colors in the steps of grand peakless Motzkin paths yields a new class of generalized versions of RNA Array III. We derive generating functions, explicit combinatorial formulas, and connections to the generalized forms of RNA Arrays I and II for the generalized RNA Array III.

When considering RNA Array III, we noticed that the columns for $k=1$ and $k=2$ appeared to correspond with counts of horizontal steps and ascents in peakless Motzkin paths given by A110236 and A114713 in~\cite{oeis}. We both prove this connection and demonstrate a broader correspondence between grand peakless Motzkin paths with colored steps and the problem of enumerating pattern statistics in peakless Motzkin paths. We are able to solve the problem of enumerating occurences of an arbitrary pattern in peakless Motzkin paths with colors. We further solve the problem of enumerating the total number of occurences of an arbitrary pattern at an arbitrary fixed height. 

For both problems, we give general formulas in terms of entries in our generalized RNA Arrays I and III, which is sufficient to easily obtain both generating functions and explicit combinatorial formulas. The problem of obtaining pattern statistics is of interest in the literature and is investigated for other classes of lattice walks in~\cite{dyckairpocketspattern,countingstrings}. Pattern statistics for peakless Motzkin paths also have applications to problems in RNA secondary structure enumeration investigated in~\cite{combinatoricssecondarystructures}. However, the problem of obtaining explicit combinatorial formulas for pattern statistics in lattice walks is not solved in~\cite{dyckairpocketspattern,countingstrings,combinatoricssecondarystructures}. Likewise, many sequences in~\cite{oeis} corresponding to pattern statistics for peakless Motzkin paths do not have explicit formulas.

Finally, we note the first column of our new RNA Array III given $j_1=j_2=j_3=1$ corresponds exactly to the Whitney number of level n of the lattice of the ideals of the fence of order 2n, given by A051286 in~\cite{oeis} and in~\cite{whitneylattices}. We define a distributive lattice on the class of lattice walks enumerated by this array and use this to give a Whitney number interpretation for RNA Array III, which may provide useful insight for future work on lattice walks or distributive lattices.

This paper is organized in the following way. In Section~\ref{sec:background}, we introduce basic definitions and results about Riordan arrays and lattice walks found in the literature which we will use later in the paper, in addition to some new definitions and preliminary results. In Section~\ref{sec:counting}, we give Riordan matrix interpretations along with explicit formulas enumerating various lattice walks and connect the enumeration of these lattice walks to the problem of calculating pattern statistics. In Section~\ref{sec:applications}, we provide examples of explicit formulas for pattern statistics which can be obtained using our results in the prior section. We also give a Whitney number interpretation for RNA Array III and investigate certain asymptotics. Finally, in Section~\ref{sec:conclusion}, we state some open problems and ideas for further work.

\section{Background and Preliminaries}\label{sec:background}
\subsection{Riordan Arrays}
\begin{definition}
    Let $f(z)=\sum_{i\ge 0} f_i z^i$ be a formal power series. Then, $[z^n] f(z)=f_n$.
\end{definition}

\begin{definition}
    Let $M=(m_{n,k})_{n,k\ge 0}$ be an infinite matrix. We call $M$ a Riordan array if and only if $m_{n,k}=[z^n] [g(z)f^k(z)]$ for some formal power series $f(z)$ and $g(z)$. We say that $M=(g(z),f(z))$.
\end{definition}

\begin{definition}
    Let $M=(g,f)$ be a Riordan array where $g$ is a formal power series of order 0 and $f$ is a formal power series of order 1. Then, we call $M$ a proper Riordan array.
\end{definition}

\begin{definition}
    Let $f$ be a formal power series. Then, $\overline{f}$ is the compositional inverse of $f$ where $f\circ\overline{f}=\overline{f}\circ f=z$.
\end{definition}

\begin{theorem}
    Proper Riordan arrays form a group called the Riordan group under standard matrix multiplication where:
    \begin{enumerate}
        \item[(i)] $(g,f)\cdot (h,v)=(g(z)h(f(z)), v(f(z)))$
        \item[(ii)] $(1,z)$ is the identity
        \item[(iii)] $(g,f)^{-1}=(\frac{1}{g(\overline{f}(z))},\overline{f}(z))$.
    \end{enumerate}
\end{theorem}

\begin{proof}
    This was shown in~\cite{originalriordanpaper} and is a foundational result for Riordan array theory.
\end{proof}

\begin{theorem}\label{thm:amatrixcharacterization}
    Let $M=(m_{n,k})_{n,k\ge 0}$ be an infinite lower triangular matrix. Then, $M$ is a proper Riordan array if and only if there exists some array $(\alpha_{i,j})_{i,j\ge 0}$ where $\alpha_{0,0}\ne 0$ and 
    \[m_{n+1,k+1}=\sum_{i\ge 0}\sum_{j\ge 0} \alpha_{i,j} m_{n-i, k+j}\]

    for all $n,k\ge 0$. We call $(\alpha_{i,j})$ an $A$-matrix for $M$.
\end{theorem}

\begin{proof}
    This is Theorem 2.5 in~\cite{alternativeriordan} and provides an alternate characterization of Riordan arrays.
\end{proof}

\begin{theorem}\label{thm:uniquemultiplier}
    Let $M=(g(z),f(z))$ be a Riordan array with $A$-matrix $(\alpha_{i,j})_{i,j\ge 0}$. Then, $f(z)$ is uniquely determined by $(\alpha_{i,j})_{i,j\ge0 }$.
\end{theorem}

\begin{proof}
    This is stated in~\cite{alternativeriordan} without proof. Alternatively, we provide the following proof. Suppose that $L=(g_1,f_1)$ and $M=(g_2, f_2)$ are both Riordan arrays with $A$-matrix $(\alpha_{i,j})_{i,j\ge 0}$. Then, for all $n,k\ge 0$,
    \[l_{n+1,k+1}=\sum_{i\ge 0}\sum_{j\ge 0} \alpha_{i,j} l_{n-i, k+j}\]

    and
    \[m_{n+1,k+1}=\sum_{i\ge 0}\sum_{j\ge 0} \alpha_{i,j} m_{n-i, k+j}\]

    by definition. So,
    \[m_{n+1,k+1}+l_{n+1,k+1} =\sum_{i\ge 0}\sum_{j\ge 0} \alpha_{i,j} (m_{n-i,k+j}+l_{n-i, k+j})\]

    which means that by Theorem~\ref{thm:amatrixcharacterization}, $M+L$ is a proper Riordan array, since it also has $A$-matrix $(\alpha_{i,j})_{i,j\ge 0}$. Since $M$ and $L$ are both Riordan arrays, $M+L$ is a Riordan array if and only if $f_1=f_2$ by Theorem 3.1 in~\cite{riordansums}. Since $M+L$ is Riordan, it follows that $f_1=f_2$. Thus, we conclude that any Riordan array with $A$-matrix $(\alpha_{i,j})_{i,j\ge 0}$ has the same multiplier $f_1=f_2=f$. This means that $(\alpha_{i,j})_{i,j\ge 0}$ uniquely determines $f$. 
\end{proof}

\begin{definition}
    Let an $A$-matrix $(\alpha_{i,j})_{i,j\ge 0}$ where only entries of the form $\alpha_{0,0}$ and $\alpha_{i,i+1}$ may be nonzero be called a $B$-matrix. We may also represent this $B$-matrix as $(\alpha_{0,0},B(z))$ where $B(z)$ is the generating function for ${(\alpha_{i,i+1})}_{i\ge 0}$.
\end{definition}

\begin{proposition}\label{prop:bmultformula}
    Let $M=(g,f)$ have a $B$-matrix $(\beta_{i,j})_{i,j\ge 0}=(\beta_{0,0}, B(z))$. Then, we conclude that $f(z)=\beta_{0,0} z + zf(z) B(zf(z))$.
\end{proposition}

\begin{proof}
    If $M=(m_{n,k})_{n,k\ge 0}$ then by definition,
    \[m_{n+1,k+1}=\beta_{0,0} m_{n,k}+\sum_{i\ge 0} \beta_{i,i+1} m_{n-i,k+1+i},\]

    Since $m_{n,k}=[z^n] gf^k$ for all $n,k\ge 0$, it follows that
    \[[z^{n+1}] gf^{k+1} = [z^n] \beta_{0,0} gf^k + \sum_{i\ge 0} \beta_{i,i+1} [z^{n-i}] gf^{k+1+i}.\]

    Next,
    \[[z^{n+1}] (zg f^{k+1} B(zf))=\sum_{i\ge 0} [z^{n+1-i}] (zgf^{k+1}) * [z^{i}] (B(zf))\]

    Where
    \[[z^i] B(zf)=\sum_{j\ge 0} [z^i] (\beta_{j,j+1} (zf)^j)\]

    And so
    \[[z^{n+1}] (zg f^{k+1} B(zf))=\sum_{j\ge 0}  \sum_{i\ge 0} [z^{n+1-i}] (zgf^{k+1}) * [z^i] (\beta_{j,j+1} (zf)^j)\]

    Which simplifies to
    \[[z^{n+1}] (zg f^{k+1} B(zf))=\sum_{j\ge 0} \beta_{j,j+1} [z^{n+1}]  (z^{j+1} g f^{k+1+j})=\sum_{j\ge 0} \beta_{j,j+1} [z^{n-j}] (gf^{k+1+j})\]

    And so we see that
    \[[z^{n+1}] gf^k = [z^{n+1}] (z \beta_{0,0} gf^{k-1}) + [z^{n+1}] (zgf^k B(zf)).\]

    From this,
    \[gf^k = z*\beta_{0,0} gf^{k-1} + zgf^k B(zf) \]

    And so
    \[f=\beta_{0,0} z + zf(z) B(zf(z))\]

    as desired.
\end{proof}

\begin{proposition}\label{prop:genbell}
    Let $M=(m_{n,k})_{n,k\ge 0}$ have a $B$-matrix $(\beta_{i,j})_{i,j\ge 0}$. Suppose that for all $n\ge 0$, $m_{n+1,0}=\sum_{i\ge 0} \beta_{i,i+1} m_{n-i,k+i}$. Then, $M=(\frac{m_{0,0} f}{\beta_{0,0} z}, f)$.
\end{proposition}

\begin{proof}
    Let $L=(m_{n-1,k-1})_{n,k\ge 0}$ where we denote $m_{-1,-1}=\frac{m_{0,0}}{\beta_{0,0}}$ and $m_{n,-1}=0$ for all $n\ge 0$. It is easy to see that $L$ obeys the required recurrence so that it has the same $B$-matrix as $M$. Then, we know that if $M=(g,f)$, $L=(\frac{m_{0,0}}{\beta_{0,0}}, f)$. Furthermore, we know that the generating function for $l_{n,1}$ is both $\frac{m_{0,0}}{\beta_{0,0}} f(z)$ and $zg(z)$. As such, we conclude that $g(z)=\frac{m_{0,0} f(z)}{\beta_{0,0} z}$.
\end{proof}

\begin{remark}
    If $\beta_{0,0}=1$, a $B$-matrix corresponds exactly to a type-I $B$-sequence as defined in~\cite{hepaper}. Propositions~\ref{prop:bmultformula} and~\ref{prop:genbell} are likewise generalizations of results in~\cite{hepaper}.
\end{remark}

\begin{definition}
    Let $M=(g,f)$ be a proper Riordan array. Then, $M$ is a pseudo-involution if $(g,-f)$ is an element of order 2 in the Riordan group.
\end{definition}

\subsection{Lattice Walks}
\begin{definition}
    Let $N=(1,1)$, $S=(1,-1)$, and $E=(1,0)$. These steps generate paths on the lattice $\mathbb{Z}^2$, which we call $NES$ paths.
\end{definition}

\begin{definition}
    We define a $(j_1,j_2,j_3)$-colored $NES$ path as an $NES$ path where we allow $j_1$ colors for $N$ steps, $j_2$ colors for $E$ steps, and $j_3$ colors for $S$ steps.
\end{definition}

\begin{notation}
In this paper, we assume some fixed choice of $j_1$, $j_2$, and $j_3$ and only consider NES paths which are $(j_1,j_2,j_3)$-colored unless otherwise stated.
\end{notation}

\begin{definition}
    A sequence of steps in a lattice walk is called a pattern. The total number of occurrences of some arbitrary fixed pattern in a class of lattice walks is called a pattern statistic.
\end{definition}

\begin{definition}
    For any $n\ge 0$, $k\ge 0$, let an $NES$ path from $(0,0)$ to $(n,k)$ which never has the pattern $NS$ and never goes below $y=0$ be called an $NES$ path of type 1.
\end{definition}

\begin{definition}
    For any $n\ge 0$, $k\ge 0$, let an $NES$ path from $(0,0)$ to $(n,k)$ which never has the pattern $SN$ and never goes below $y=0$ be called an $NES$ path of type 2.
\end{definition}

\begin{definition}
    For any $n\ge 0$, $k\ge 0$, let an $NES$ path from $(0,0)$ to $(n,k)$ which never has the pattern $NS$ be called an $NES$ path of type 3.
\end{definition}

\begin{definition}
    Let $w$ be an $NES$ path. Let $w^*$ be the path constructed by reversing $w$ and flipping $N$ steps and $S$ steps. We call $w^*$ the dual path of $w$.
\end{definition}

\begin{proposition}\label{prop:gpathends}
    Let $w$ be an arbitrary NES type 3 path from $(0,0)$ to $(m,j)$ not beginning with $S$. Then, NES type 3 paths from $(0,0)$ to $(n,k)$ ending in $w$ are in bijection with NES type 3 paths from $(0,0)$ to $(n-m,k-j)$.
\end{proposition}

\begin{proof}
    Let $p$ be some NES type 3 path from $(0,0)$ to $(n,k)$ ending in $w$. Then, remove $w$ from the end of $p$. This gives us a path from $(0,0)$ to $(n-m,k-j)$, which remains NES type 3. Conversely, let $q$ be some NES type 3 path from $(0,0)$ to $(n-m,k-j)$. Then, append $w$ to the end of $q$. Since $w$ does not start with $S$, this cannot create a peak, and instead yields an NES type 3 path from $(0,0)$ to $(n,k)$. It is easy to see that these two processes are inverses, so we have constructed a bijection. As $g_{n-m,k-j}$ counts NES type 3 paths from $(0,0)$ to $(n-m,k-j)$, it must likewise count NES type 3 paths from $(0,0)$ to $(n,k)$ ending in $w$. 
\end{proof}

\begin{proposition}\label{prop:spathends}
    Let $w$ be an arbitrary NES type 3 path from $(0,0)$ to $(m,j)$ which always stays at $y\ge j-k$ and does not begin with $S$. Then, NES type 1 paths from $(0,0)$ to $(n,k)$ ending in $w$ are in bijection with NES type 1 paths from $(0,0)$ to $(n-m,k-j)$.
\end{proposition}

\begin{proof}
    Consider some arbitrary NES type 1 path from $(0,0)$ to $(n-m,k-j)$. Then, append $w$ to the end of this path. This gives us a path from $(0,0)$ to $(n,k)$, and since $w$ does not start with $S$, we did not add any peaks. Furthermore, appending $w$ to $p$, since $w$ stays at $y\ge j-k$ when considered as starting at $y=0$, starting at $y=k-j$, we find that appending $w$ still gives us a path staying at $y\ge 0$. As such, we have an NES type 1 path from $(0,0)$ to $(n,k)$.

    Conversely, consider an NES type 1 path from $(0,0)$ to $(n,k)$ ending in $w$. Removing $w$ from the end of this path gives us an NES type 1 path, as we still have no peaks and stay at $y\ge 0$ in all left prefixes by definition. Furthermore, this path is from $(0,0)$ to $(n-m,k-j)$.

    It is easy to see that these two processes are inverses. As such, we have a bijection, and the number of NES type 1 paths from $(0,0)$ to $(n,k)$ ending in $w$ is equal to the number of NES type 1 paths from $(0,0)$ to $(n-m,k-j)$. This means that such paths are likewise counted by $s_{n-m,k-j}$.
\end{proof}

\begin{proposition}\label{prop:dualbijection}
    The map $w\mapsto w^*$ induces a bijection between $(j_1,j_2,j_3)$-colored NES type 3 paths from $(0,0)$ to $(n,k)$ and $(j_3,j_2,j_1)$-colored NES type 3 paths from $(0,0)$ to $(n,-k)$.
\end{proposition}

\begin{proof}
    It is easy to see that for any NES path $w$, $(w^*)^*=w$, so it is sufficient to show that for any $(j_1,j_2,j_3)$-colored NES type 3 path $p$ from $(0,0)$ to $(n,k)$, $p^*$ is a $(j_3,j_2,j_1)$-colored NES type 3 path from $(0,0)$ to $(n,-k)$.

    First, we see that any pattern of the form $NS$, when $N$ and $S$ steps are flipped and the order of steps is reversed, is mapped to $NS$. As such, given an $NES$ type 3 path $p$ where the pattern $NS$ is forbidden, $NS$ likewise does not occur in $p^*$. Likewise, $p$ is a path from $(0,0)$ to $(n,k)$ iff in its linear form, there exist $k$ more $N$ steps than $S$ steps. Since $p^*$ has $N$ and $S$ steps flipped, $p^*$ has $k$ more $S$ steps than $N$ steps. As a result, $p^*$ describes an NES type 3 path from $(0,0)$ to $(n,-k)$. Finally, we note that while $p^*$ has $N$ and $S$ steps flipped, and $p\mapsto p^*$ reverses the ordering of steps, $p\mapsto p^*$ does not alter the color associated with each step. As such, while $E$ steps still have $j_2$ colors, $N$ and $S$ steps in $p^*$ have $j_3$ and $j_1$ colors respectively. Thus, we are able to conclude that given a $(j_1,j_2,j_3)$-colored NES type 3 path $p$ from $(0,0)$ to $(n,k)$, $p^*$ is a $(j_3,j_2,j_1)$-colored NES type 3 path from $(0,0)$ to $(n,-k)$.
\end{proof}

\begin{proposition}\label{prop:sdualbijection}
    The map $w\mapsto w^*$ induces a bijection between $(j_1,j_2,j_3)$-colored NES type 1 paths from $(0,0)$ to $(n,0)$ and $(j_3,j_2,j_1)$-colored NES type 1 paths from $(0,0)$ to $(n,0)$.
\end{proposition}

\begin{proof}
    As in Proposition~\ref{prop:dualbijection}, it is sufficient to show that given some $(j_1,j_2,j_3)$-colored NES type 1 path $p$ from $(0,0)$ to $(n,0)$, $p^*$ is a $(j_3,j_2,j_1)$-colored NES type 1 path from $(0,0)$ to $(n,0)$. Let $p$ be some arbitrary $(j_1,j_2,j_3)$-colored NES type 1 path $p$ from $(0,0)$ to $(n,0)$. By definition, $p$ is also an NES type 3 path, so by Proposition~\ref{prop:dualbijection}, $p^*$ is a $(j_3,j_2,j_1)$-colored NES type 3 path.

    Next, suppose for the sake of contradiction that $p^*$ is not an NES type 1 path. Then, $p^*$ must reach $y=-1$. Take the first point where this occurs. Then, $p^*=qr$ where $q$ is an NES type 3 path with endpoint height $-1$ and $r$ is an NES type 3 path with endpoint height 1. Now, $p=(p^*)^*=r^* q^*$. Since $q$ has endpoint height $-1$, by Proposition~\ref{prop:dualbijection}, $q^*$ has endpoint height 1. Since $p$ ends at $y=0$, $q^*$ likewise ends at $y=0$. As $q^*$ has endpoint height 1, $q^*$ must begin at $y=-1$. However, $p$ is an NES type 1 path, so it cannot reach $y=-1$, and we have a contradiction.

    Thus, by contradiction, we conclude that $p^*$ is an NES type 1 path, and we are done.
\end{proof}

\begin{definition}
    A sequence of steps $p$ forms a dominating sequence if all nonempty left prefixes of $p$ have strictly more $N$ steps than $S$ steps.
\end{definition}

\begin{definition}
    Let $p$ and $q$ be sequences of steps. We say that $p$ and $q$ are circular rotations of one another if there exist some sequences of steps $r$ and $s$ such that $p=rs$ and $q=sr$.
\end{definition}

\begin{proposition}\label{prop:cyclelemma}
    Given a sequence of steps $p$ with exactly $h$ more $N$ steps than $S$ steps, there are exactly $h$ circular rotations of $p$ which form dominating sequences.
\end{proposition}

\begin{proof}
    This is a statement of a well-known combinatorial result known as the Cycle Lemma which was proven in~\cite{cyclelemmapaper}.
\end{proof}

\section{Path Enumeration}\label{sec:counting}
\subsection{Peakless Paths}
\begin{definition}
    Let $s_{n,k}^{j_1,j_2,j_3}$ be the number of $(j_1,j_2,j_3)$-colored NES type 1 paths starting at $(0,0)$ and ending at $(n,k)$.
\end{definition}

\begin{definition}
    Let $S_{j_1,j_2,j_3}=(s_{n,k}^{j_1,j_2,j_3})_{n,k\ge 0}$ be an infinite matrix.
\end{definition}

\begin{notation}
    We let $S=(s_{n,k})_{n,k \ge 0}$ refer to $S_{j_1,j_2,j_3}$ for our arbitrary fixed choice of $j_1$, $j_2$, and $j_3$.
\end{notation}

\begin{theorem}\label{thm:srecurrence}
    $S$ obeys the following recurrences:
    \begin{enumerate}
        \item[(i)] For $n\ge 1$, $s_{n,0}=j_2\sum_{i\ge 0} {(j_3)}^i * s_{n-1-i,k+i}$.
        \item[(ii)] For $n,k\ge 1$, $s_{n,k}=j_1*s_{n-1,k-1}+j_2\sum_{i\ge 0} (j_3)^i * s_{n-1-i,k+i}$.
    \end{enumerate}
\end{theorem}

\begin{proof}
    Given a path counted by $s_{n,k}$, it either ends on $N$ or ends on $ES^i$ for some $i$. By Proposition~\ref{prop:spathends}, for $k\ge 1$, there are $s_{n-1,k-1}$ NES type 1 paths from $(0,0)$ to $(n,k)$ with a final step of $N$ given a fixed color of $N$. As there are $j_1$ colors for $N$, we then have $j_1*s_{n-1,k-1}$ paths ending in $N$.

    Similarly, for any $i\ge 0$, there are $j_2*{(j_3)}^i$ different possible colorings for $ES^i$. Again, by Proposition~\ref{prop:spathends}, there are $j_2*{(j_3)}^i*s_{n-1-i,k+i}$ paths ending in $ES^i$. Thus, for $k\ge 1$, 
    \[s_{n,k}=j_1*s_{n-1,k-1}+j_2\sum_{i\ge 0} (j_3)^i * s_{n-1-i,k+i}.\]

    Meanwhile, for $k=0$, NES type 1 paths from $(0,0)$ to $(n,k)$ can only end in $ES^i$ for some $i\ge 0$, since ending on $N$ would mean that the path reached $y=-1$, which is impossible. As such, we do not include the $s_{n-1,k-1}$ term and instead find that
    \[s_{n,0}=j_2\sum_{i\ge 0} {(j_3)}^i * s_{n-1-i,k+i}\]

    as desired.
\end{proof}

\begin{corollary}
    $S$ is the Riordan matrix $(s(z),j_1*zs(z))$.
\end{corollary}

\begin{proof}
    The recurrence in Theorem~\ref{thm:srecurrence} gives us a $B$-matrix for $S$. Moreover, we know that 
    \[s_{n,0}=j_2\sum_{i=0}^{n-1} (j_3)^i * s_{n-1-i,k+i}\]

    and so by Proposition~\ref{prop:genbell}, $S=(\frac{f}{j_1 z}, f)$ for some $f$. Additionally, $S=(s,f)$, so $f(z)=j_1 zs(z)$. As such, $S=(s(z),j_1zs(z))$. 
\end{proof}

\begin{proposition}
    For arbitrary $j_2,j_3\ge 1$, $S_{1,1,j_3}^{j_2}=S_{1,j_2,j_3}$.
\end{proposition}

\begin{proof}
    We know by Theorem~\ref{thm:srecurrence} that $S_{1,1,j_3}$ is in the Bell subgroup and has a $B$-matrix of the form $(1,\frac{1}{1-j_3 z})$. Then, $S_{1,1,j_3}$ equivalently has a $\Delta$-sequence generated by $\frac{1}{1-j_3 z}$ where we define a $\Delta$-sequence as in~\cite{deltasequence}. By Theorem 4.2 in~\cite{deltasequence}, $S_{1,1,j_3}^{j_2}$ is a matrix in the Bell subgroup with $\Delta$-sequence generated by $\frac{j_2}{1-j_3 z}$. This means that $S_{1,1,j_3}^{j_2}$ is equivalently a Bell matrix with a $B$-matrix $(1,\frac{j_2}{1-j_3 z})$. This is the same $B$-matrix that $S_{1,j_2,j_3}$ has, so by Theorem~\ref{thm:uniquemultiplier}, $S_{1,j_2,j_3}$ and $S_{1,1,j_3}^{j_2}$ have the same multiplier. Since both are also in the Bell subgroup, they must be the same matrix.
\end{proof}

\begin{remark}
    This generalizes a result in~\cite{evanspaper}.
\end{remark}

\begin{proposition}\label{prop:pseudoinvolution}
    For arbitrary $j_1,j_2,j_3\ge 1$, $S_{j_1,j_2,j_3}$ is a pseudo-involution iff $j_1=1$.
\end{proposition}

\begin{proof}
    If $j_1=1$, $S$ is a Bell-type Riordan matrix with both types of $B$-sequence. As such, $S$ is a pseudo-involution by~\cite{hepaper}.

    If $j_1 > 1$, then $S=(s(z),j_1*zs(z))$. We know that $S$ is a pseudo-involution iff $(s(z),-j_1*zs(z))$ has order 2. However, the diagonal of $(s(z),-j_1*zs(z))$ has entries which are powers of $-j_1$, so the diagonal of $(s(z),-j_1*zs(z))$ has entries which are powers of $j_1^2$. Since $j_1 > 1$, this cannot be the identity, so $S$ is not a pseudo-involution.
\end{proof}

\begin{corollary}
    $S^{-1}=(s(-\frac{z}{j_1}),\frac{z}{j_1}s(-\frac{z}{j_1}))$.
\end{corollary}
\begin{proof}
    Consider $(s(z), zs(z))$. This corresponds to 1 color for $N$ and $j_1*j_3$ colors for $S$, so by Proposition~\ref{prop:pseudoinvolution}, this matrix is a pseudo-involution. As such, by Proposition 2 in~\cite{algebraicriordanstructure}, 
    \[(s(z),zs(z))^{-1}=(s(-z),zs(-z)).\] 
    
    Now, we know that $S=(s(z),zs(z))*(1,j_1 z)$, so 
    \[S^{-1}=(1,\frac{1}{j_1} z)*(s(z),zs(z))^{-1}=(s(-\frac{z}{j_1}),\frac{z}{j_1}s(-\frac{z}{j_1})).\]
\end{proof}

\begin{proposition}\label{prop:generatingexpression}
    $j_1 j_3 * z^2 s^2(z)-j_1 j_3*z^2 s(z)+j_2 *zs(z)-s(z)+1=0$.
\end{proposition}

\begin{proof}
    For $S$, our $B$-matrix is $(j_1,\frac{j_2}{1-j_3 z})$. By Proposition~\ref{prop:bmultformula}, 
    \[j_1*zs(z)=j_1 z + j_1 * z^2 s(z) *\frac{j_2}{1-j_3*j_1*z^2 s(z)}\]

    And so it follows that
    \[s(z)=1+zs(z)*\frac{j_2}{1-j_3*j_1*z^2 s(z)}\]

    Which simplifies to
    \[s(z)-j_1 * j_3 * z^2 s^2(z)=1-j_3*j_1*z^2 s(z)+zs(z)*j_2\]

    And we then find that
    \[j_1 j_3 * z^2 s^2(z)-j_1 j_3*z^2 s(z)+j_2 *zs(z)-s(z)+1=0.\]
\end{proof}

\begin{remark}
     This means that $s(z)=1+zs(z)*\frac{j_2}{1-j_3 j_1 z^2 s(z)}$ is uniquely defined by the values of $j_1*j_3$ and $j_2$. This is consistent with the bijection given in Proposition~\ref{prop:sdualbijection}. It also makes sense when we consider it using the RNA secondary structure interpretation, as colors for $N$ and $S$ can be interpreted as being colors for the arc between a pair of bonded nucleotides. Furthermore, this expression generalizes an equation given in~\cite{evanspaper}.
\end{remark}

\begin{theorem}
    For $n\ge 1$, $h\ge 0$, 
    \[s_{n+h,h}=(h+1)\sum_{k=\lceil \frac{n+1}{2}\rceil}^n \frac{(j_1)^{n-k+h} (j_2)^{2k-n} (j_3)^{n-k}}{k}{k\choose n-k}{k+h\choose n-k+h+1}.\]
\end{theorem}

\begin{proof}
    First, let $k$ be the total number of $S$ and $E$ steps. Setting aside the $h$ $N$ steps that do not have a corresponding $S$ step, we have $n$ total steps. Then, every $S$ step has a corresponding $N$ step, and we know that we must have at least one $E$ step, as $n\ge 1$ and we must always be both peakless and at $y\ge 0$. As such, $\frac{n}{2}\le k\le n$. It follows that $\lceil \frac{n+1}{2}\rceil \le k\le n$.

    Now, consider some fixed $k$. We have $k+h$ steps of type $N$ or $E$, so we have $n-k$ steps of type $S$. Then, we must have $n-k+h$ steps of type $N$ for an ending height of $N$, and $2k-n$ steps of type $E$ remain. Furthermore, since words of the form $NS$ are not permitted, in order to construct our path, it is sufficient to order $N$ and $E$ steps, as well as $S$ and $E$ steps. This is because if some $N$ and $S$ have no $E$ steps between them, the $S$ must precede the $N$ to avoid a peak. This gives us ${k+h\choose 2k-n}$ options for the ordering of $N$ and $E$. To order $S$ and $E$, we first place an $E$ at the beginning of the ordering, as our path cannot begin with an $S$ step. This gives us ${k-1\choose n-k}$ options. So, we have ${k+h\choose 2k-n}{k-1\choose n-k}$ options in total.

    Next, we note that we have overcounted, as we are considering paths that go below $y=0$. We now define an equivalence relation on the paths that we enumerated above. First, given some path, add an $N$ step to the beginning. Since no paths begin with $S$, this does not create a peak. Next, we say that the two initial paths are equivalent iff after adding $N$ to the beginning of each, we obtain paths which are circular rotations of one another. Each of these equivalence classes has $n-k+h+1$ members. 
    
    Now, in each equivalence class, after we remove the $N$ step at the beginning, we obtain a path that stays at $y\ge 0$ iff the previous path always stays at $y\ge 1$. Equivalently, every left prefix of the path with an $N$ step added to the beginning must have more $N$ steps than $S$ steps present. By Proposition~\ref{prop:cyclelemma}, as we have $h+1$ more $N$ steps than $S$ steps, each equivalence class has exactly $h+1$ paths that we should be counting. Note that for this purpose we may ignore $E$ steps in each path.
    
    This gives us
    \[\frac{h+1}{n-k+h+1}{k-1\choose n-k}{k+h\choose 2k-n}=\frac{h+1}{k}{k\choose n-k}{k+h\choose n-k+h+1}\]

    options. Finally, we must color each of the steps in the paths that we have constructed. Given this fixed $k$, we have $n-k+h$ steps of type $N$, so there are ${(j_1)}^{n-k+h}$ options for colors of $N$ steps. Likewise, we have factors of ${(j_2)}^{2k-n}$ and ${(j_3)}^{n-k}$. Adding in these additional factors to our expression gives us
    \[\frac{(h+1){(j_1)}^{n-k+h}{(j_2)}^{2k-n}{(j_3)}^{n-k}}{k}{k\choose n-k}{k+h\choose n-k+h+1}\]

    And when we sum across $\lceil \frac{n+1}{2}\rceil\le k\le n$, we obtain the desired expression for $s_{n+h,h}$.
\end{proof}

\begin{remark}
    This generalizes a formula given for A097724 in~\cite{oeis}.
\end{remark}

\subsection{Valleyless Paths}
\begin{definition}
    Let $t_{n,k}$ be the number of NES type 2 paths from $(0,0)$ to $(n,k)$.
\end{definition}

\begin{definition}
    Let $T=(t_{n,k})_{n,k\ge 0}$ be an infinite matrix.
\end{definition}

\begin{proposition}\label{prop:peakvalleybijection}
    The number of NES type 1 paths from $(0,0)$ to $(n+1,k)$ which do not end in $N$ is equal to $j_2*t_{n,k}$.
\end{proposition}

\begin{proof}
    First, consider $E$ with some color. Then, take an NES type 2 path from $(0,0)$ to $(n,k)$ and add $E$ with this designated color to the beginning. Then, every ascent must be preceded by an $E$ step, since $SN$ is a forbidden pattern. As such, take the $E$ step preceding each ascent and step it to the end of that ascent. Then, all ascents are followed by an $E$ step, which gives us an NES type 1 path not ending in $N$ from $(0,0)$ to $(n+1,k)$. Furthermore, the first $E$ step in this NES type 1 path has our designated color. It is easy to see that by swapping ascents with $E$ steps following them and then removing the $E$ step at the beginning of the path, we create a bijection between NES type 2 paths from $(0,0)$ to $(n,k)$ and NES type 1 paths from $(0,0)$ to $(n+1,k)$ not ending in $N$ and with a first $E$ step of our designated color.

    Now, we have $j_2$ options for such $E$ steps, so if we count all NES type 1 paths from $(0,0) $to $(n+1,k)$ not ending in $N$, this is equal to $j_2*t_{n,k}$.
\end{proof}

\begin{remark}
    This generalizes Theorem 3.1 in~\cite{asamoahearlypaper}.
\end{remark}

\begin{proposition}
    $j_2*t_{n,k}=s_{n+1,k}-j_1*s_{n,k-1}$
\end{proposition}

\begin{proof}
    By Proposition~\ref{prop:peakvalleybijection}, we know that $j_2*t_{n,k}$ enumerates NES type 1 paths from $(0,0)$ to $(n+1,k)$ not ending in $N$. By Proposition~\ref{prop:spathends}, $j_1*s_{n,k-1}$ enumerates NES type 1 paths from $(0,0)$ to $(n+1,k)$ ending in $N$, and $s_{n+1,k}$ enumerates all such NES type 1 paths by definition. As such, NES type 1 paths from $(0,0)$ to $(n+1,k)$ not ending in $N$ are enumerated by both expressions, so the two expressions must be equal.
\end{proof}

\begin{corollary}
    $T$ is the Riordan array $(\frac{s(z)-1}{j_2*z},j_1*zs(z))$.
\end{corollary}

\begin{proposition}
    $j_2*t_{n,k}-j_2*s_{n,k}=j_3*t_{n-1,k+1}$.
\end{proposition}

\begin{proof}
     By Proposition~\ref{prop:peakvalleybijection}, we know that $j_2*t_{n,k}$ enumerates NES type 1 paths from $(0,0)$ to $(n+1,k)$ not ending in $N$. By Proposition~\ref{prop:spathends}, $j_2*s_{n,k}$ enumerates NES type 1 paths from $(0,0)$ to $(n+1,k)$ ending in $E$. So, $j_2*t_{n,k}-j_2*s_{n,k}$ enumerates NES type 1 paths from $(0,0)$ to $(n+1,k)$ ending in $S$. NES type 1 paths from $(0,0)$ to $(n+1,k)$ ending in $S$ uniquely correspond to NES type 1 paths from $(0,0)$ to $(n,k+1)$ not ending in $N$ along with a provided color for an $S$ step to add at the end. There are $j_3$ options for colors for $S$ and $t_{n-1,k+1}$ NES type 1 paths from $(0,0)$ to $(n,k+1)$ not ending in $N$ by Proposition~\ref{prop:peakvalleybijection}. As such, we have $j_3*t_{n-1,k+1}$ options. Since these two expressions enumerate the same class of objects, they must be equal.
\end{proof}

\begin{remark}
    This means that the difference between valleyless and peakless paths given any options for colors always forms a Riordan array.
\end{remark}

\begin{corollary}
    $j_1 j_3 * z^2 s^2(z)-j_1 j_3*z^2 s(z)+j_2 *zs(z)-s(z)+1=0$.
\end{corollary}

\begin{proof}
    Taking $k=0$, $t_{n,0}-j_2*s_{n,0}$ has a generating function of $\frac{s(z)-1}{z}-j_2*s(z)$. Meanwhile, $j_3*t_{n-1,1}$ has a generating function of $\frac{s(z)-1}{z}*j_1*j_3*z^2 s(z)$. Setting these equal, we get the identity
    \[s(z)-1-j_2*zs(z)=j_1*j_3*z^2 s^2(z)-j_1*j_3*z^2 s(z)\]

    And this results in
    \[j_1 j_3 * z^2 s^2(z)-j_1j_3*z^2 s(z)+ j_2*zs(z)-s(z)+1=0\]

    as desired.
\end{proof}

\begin{remark}
    This gives an alternate proof and a combinatorial interpretation for the result in Proposition~\ref{prop:generatingexpression}.
\end{remark}

\subsection{Grand Paths}\label{sec:grand}
\begin{definition}
    Let $g_{n,k}$ be the number of NES type 3 paths from $(0,0)$ to $(n,k)$.
\end{definition}

\begin{definition}
    Let $G=(g_{n,k})_{n,k\ge 0}$ be an infinite matrix.
\end{definition}

\begin{proposition}
    For $n,h\ge 0$, 
    \[g_{n+h,h}=\sum_{k=\lceil \frac{n}{2}\rceil}^{n} (j_1)^{n-k+h} (j_2)^{2k-n} (j_3)^{n-k} {k+h\choose 2k-n}{k\choose 2k-n}.\]
\end{proposition}

\begin{proof}
    Enumerate the total number of $N$ and $E$ steps by $k+h$. We know that we have $h$ steps of type $N$ which do not correspond to an $S$ step. The remaining steps of type $N$, along with the steps of type $E$, are enumerated by $k$. These $N$ steps each correspond to $S$ steps, which are counted by $n-k$. As such, we find that $n-k\le k$, meaning that $2k\ge n$, and so $k\ge \lceil \frac{n}{2}\rceil$. Furthermore, $k+h\le n+h$, so $k\le n$.

    Now, consider some fixed $k$. We then have $n-k+h$ $N$ steps, $2k-n$ $E$ steps, and $n-k$ $S$ steps. To fully specify any given path, since we have no peaks, it is sufficient to specify only the orderings of $N$ and $E$ steps, as well as of $S$ and $E$ steps. To order $N$ and $E$ steps, we obtain ${k+h\choose 2k-n}$ options. We obtain ${k\choose 2k-n}$ options for orderings of $S$ and $E$ steps.

    Finally, we describe colorings of our steps, given a particular ordering of steps, with a factor of
    \[(j_1)^{n-k+h} (j_2)^{2k-n} (j_3)^{n-k}\]

    based on the number of each type of step. So, for each $k$, we obtain
    \[(j_1)^{n-k+h} (j_2)^{2k-n} (j_3)^{n-k}{k+h\choose 2k-n}{k\choose 2k-n}\]

    and we obtain the desired formula by taking a sum over $\lceil\frac{n}{2}\rceil\le k\le n$.
\end{proof}

\begin{theorem}\label{thm:grecurrence}
    For $n\ge 1$ and $k\ge 0$, $g_{n,k}=j_1*g_{n-1,k-1}+j_2\sum_{i=0}^{n-1} (j_3)^i * g_{n-1-i,k+i}$
\end{theorem}

\begin{proof}
    Consider an arbitrary NES type 3 path $p$ from $(0,0)$ to $(n,k)$. We may consider cases where $p$ ends in $N$, $E$ or $S$. Now, if $p$ ends in $S$, since $k\ge 0$, it is impossible that $p$ is entirely composed of $S$ steps. Since $p$ may not contain the pattern $NS$, if $p$ ends in $S$, then $p$ must end in $ES^j$ for $j\ge 1$. In this way, $p$ either ends in $N$ or $ES^j$ for some $j\ge 0$.

    By Proposition~\ref{prop:gpathends}, if $w=N$ for some fixed color, then there are $g_{n-1,k-1}$ NES type 3 paths from $(0,0)$ to $(n,k)$ ending in $w$. There are $j_1$ colors for $N$, so this gives us $j_1*g_{n-1,k-1}$ paths ending in $N$. Likewise, for any $i\ge 0$, $w=ES^i$ has $j_2*j_3^i$ options for colors, so again by Proposition~\ref{prop:gpathends}, for each $i$, there exist $j_2*j_3^i*g_{n-1-i,k+i}$ NES type 3 paths from $(0,0)$ to $(n,k)$ ending in $ES^i$. In total, we find that there are 
    \[g_{n,k}=j_1*g_{n-1,k-1}+j_2\sum_{i=0}^{n-1} {(j_3)}^i g_{n-1-i,k+i}\]

    NES type 3 paths from $(0,0)$ to $(n,k)$ as desired.
\end{proof}

\begin{theorem}
    $G$ is the Riordan matrix $(\frac{s(z)}{1-j_1 j_3 * z^2 s^2(z)},j_1*zs(z))$.
\end{theorem}

\begin{proof}
    First, consider the matrix $(g_{n,k})_{n,k\ge -1}$, where $g_{-1,-1}=\frac{1}{j_1}$ and $g_{-1,k}=0$ for all $k\ge 0$ for ease of notation. Then, the recurrence in Theorem~\ref{thm:grecurrence} holds for all $n\ge 1$ and $k\ge 0$, and $g_{0,0}=j_1*g_{-1,-1}$. As a result, $(g_{n,k})_{n,k\ge -1}$ has the same $B$-matrix as $S$ by definition. Using Theorem~\ref{thm:amatrixcharacterization} and Theorem~\ref{thm:uniquemultiplier}, we conclude that $(g_{n,k})_{n,k\ge -1}$ is a Riordan matrix with multiplier $j_1*zs(z)$.

    Next, let $g(z)$ be the generating function for $(g_{n,0})_{n\ge 0}$. Then, $z*g(z)$ is the generating function for $(g_{n,0})_{n\ge -1}$. Since $(g_{n,k})_{n,k\ge -1}$ is a Riordan matrix with multiplier $j_1*zs(z)$, $zg(z)*j_1 zs(z)$ is the generating function for $(g_{n,1})_{n\ge -1}$ and $\frac{g(z)}{j_1 s(z)}$ is the generating function for $(g_{n,-1})_{n\ge -1}$.

    Now, by Proposition~\ref{prop:dualbijection}, for $n\ge 0$, $g_{n,-1}$ is equal to the number of $(j_3,j_2,j_1)$-colored NES type 3 paths from $(0,0)$ to $(n,1)$. In addition, by Proposition~\ref{prop:dualbijection}, we know that $g_{n,0}$ is equal to the number of $(j_3,j_2,j_1)$-colored NES type 3 paths from $(0,0)$ to $(n,0)$. As a result, by the same logic as above, $g(z)*j_3 zs(z)$ is the generating function for $(g_{n,-1})_{n\ge 0}$. Note that this holds because $s(z)$ is also the generating function for $s_{n,0}^{j_3,j_2,j_1}$ by Proposition~\ref{prop:sdualbijection}. Then, since $g_{-1,-1}=\frac{1}{j_1}$, we find that the generating function for $(g_{n,-1})_{n\ge -1}$ is $\frac{1}{j_1}+zg(z)*j_3 zs(z)$.

    Since $(g_{n,k})_{n,k\ge -1}$ is a Riordan array with multiplier $j_1*zs(z)$, we now find that
    \[(\frac{1}{j_1}+zg(z)*j_3 zs(z))*j_1^2 z^2 s^2(z) = zg(z)*j_1 zs(z).\]

    Simplifying,
    \[g(z)*zs(z)*(j_1-j_3*j_1^2 *z^2 s^2(z))=j_1*z s^2 (z)\]

    And so
    \[g(z)=\frac{s(z)}{1-j_1 j_3 * z^2 s^2(z)}.\]

    Finally, we note that $G=(g_{n,k})_{n,k\ge 0}$ has the same $B$-matrix as $(g_{n,k})_{n,k\ge -1}$, and its initial column is generated by $g(z)$. So, $G=(g(z),j_1*zs(z))=(\frac{s(z)}{1-j_1 j_3 * z^2 s^2(z)}, j_1*zs(z))$ as desired.
\end{proof}

\begin{corollary}\label{cor:alternateexplicit}
    $s_{n,k}=g_{n,k}-\frac{j_3}{j_1} g_{n,k+2}$.
\end{corollary}

\begin{remark}
    This gives us an alternate explicit formula for $s_{n,k}$.
\end{remark}

\begin{proposition}\label{prop:lowestpointenumeration}
    The number of NES type 3 paths from $(0,0)$ to $(n,k)$ with a lowest point at $y=-j$ is equal to ${(\frac{j_3}{j_1})}^j s_{n,k+2j}$.
\end{proposition}

\begin{proof}
    First, we ignore colors and effectively assume that $j_1=j_2=j_3=1$. We show that in this case, there exists a bijection between NES type 3 paths from $(0,0)$ to $(n,k)$ with a lowest point at $y=-j$, and NES type 1 paths from $(0,0)$ to $(n,k+2j)$.

    Let $pq$ be an NES type 3 path from $(0,0)$ to $(n,k)$ with a lowest point at $y=-j$. We let $p$ be the largest left prefix of $pq$ which ends at $y=-j$. We note that $p$ must be immediately followed by either $E$ or $N$, as it reaches the lowest point of $pq$. As such, $p^* q$ cannot have a peak at the end of $p^*$ and the beginning of $q$. Furthermore, $p$ stays at $y\ge -j$ while ending at $y=-j$, so $p^*$ must never go below its starting height. Similarly, $q$ must never go below its starting height of $y=-j$ in $pq$, as $-j$ is the lowest height which $pq$ reaches. As such, in $p^* q$, $q$ starts at $y=j$ and stays at $y\ge j$. This is sufficient to conclude that $p^* q$ is an NES type 1 path. Note that $p$ has endpoint height $-j$ and $pq$ has endpoint height $k$, so $q$ must have endpoint height $j+k$. This means that $p^* q$ has endpoint height $k+2j$.
    
    Likewise, let $pq$ be an NES type 1 path from $(0,0)$ to $(n,k+2j)$ where $p$ is the largest left prefix which reaches $y=j$. We know that $q$ starts at $y=j$, ends at $y=k+2j$, and never reaches $y=j$ again, so it cannot start on $S$. This means that $p^* q$ does not have a peak. Furthermore, $p$ always stays at $y\ge 0$ and ends at $y=j$, meaning that $p$ always stays at most $j$ below its ending height. As such, $p^*$ stays at most $j$ below its starting height, making this an NES type 3 path with a lowest point at $y=-j$. Now, $p^* q$ is an NES type 3 path as well, with a lowest point at $y=-j$ and an ending height of $y=k$ since $q$ cannot go below its starting height as we showed above.

    It is easy to see that these two processes are inverses, so we have a bijection. 
    
    Now, we will consider this problem with colored steps. We may define an equivalence class on $(j_1,j_2,j_3)$-colored NES type 3 paths from $(0,0)$ to $(n,k)$ with a lowest point at $y=-j$ by removing colors from steps. Each equivalence class then corresponds to an uncolored NES type 3 path with a lowest point at $y=-j$. Likewise, we may define an equivalence class on NES type 1 paths from $(0,0)$ to $(n,k+2j)$. Each of these equivalence classes correspond to an uncolored NES type 1 path. The bijection described between uncolored paths above then induces a bijection between these sets of equivalence classes on colored paths.
    
    Consider the equivalence classes of two corresponding uncolored paths $pq$ and $p^* q$. We know that $pq$ has $a$ $N$ steps, $b$ $E$ steps, and $c$ $S$ steps. Then, $p^* q$ has $a+j$ $N$ steps, $b$ $E$ steps, and $c-j$ $S$ steps. We then find that $pq$ has an equivalence class of $j_1^a j_2^b j_3^c$ colored paths and $p^* q$ has an equivalence class of $j_1^{a+j} j_2^b j_3^{c-j}$ colored paths. We note that $j_1^a j_2^b j_3^c={(\frac{j_3}{j_1})}^j j_1^{a+j} j_2^b j_3^{c-j}$.
    
    So, our equivalence classes of NES type 3 paths with a lowest point at $y=-j$ always differ from their corresponding equivalence classes of NES type 1 paths in size by a factor of $(\frac{j_3}{j_1})^j$. We note that our first set of equivalence classes creates a partition of $(j_1,j_2,j_3)$-colored NES type 3 paths from $(0,0)$ to $(n,k)$ with a lowest point at $y=-j$. Likewise, our second set of equivalence classes partitions $(j_1,j_2,j_3)$-colored NES type 1 paths from $(0,0)$ to $(n,k+2j)$. This is sufficient to conclude that the set of $(j_1,j_2,j_3)$-colored NES type 3 paths from $(0,0)$ to $(n,k)$ with a lowest point at $y=-j$ differs in size from $(j_1,j_2,j_3)$-colored NES type 1 paths from $(0,0)$ to $(n,k+2j)$ by the same factor. As such, there exist ${(\frac{j_3}{j_1})}^j s_{n,k+2j}$ $(j_1,j_2,j_3)$-colored NES type 3 paths from $(0,0)$ to $(n,k)$ with a lowest point at $y=-j$.
\end{proof}

\begin{corollary}
    $g_{n,k}=\sum_{j\ge 0} {(\frac{j_3}{j_1})}^j s_{n,k+2j}$
\end{corollary}

\begin{proof}
    We know that any NES type 3 path from $(0,0)$ to $(n,k)$ has a lowest point at $y=-j$ for some unique $j\ge 0$. From this, our desired summation immediately follows.
\end{proof}

\begin{remark}
    This gives us a combinatorial interpretation for Corollary~\ref{cor:alternateexplicit}.
\end{remark}

\subsection{Pattern Statistics}
\begin{definition}
    Let $s_{n,k}^*$ denote the total number of $(j_3,j_2,j_1)$-colored NES type 1 paths starting at $(0,0)$ and ending at $(n,k)$.
\end{definition}

\begin{definition}
    Let $g_{n,k}^*$ denote the total number of $(j_3,j_2,j_1)$-colored NES type 3 paths starting at $(0,0)$ and ending at $(n,k)$.
\end{definition}

\begin{remark}
    When $j_1=j_3$, $g^*_{n,k}=g_{n,k}$ and $s_{n,k}^*=s_{n,k}$. For the remainder of this paper, we will refer to $(j_3,j_2,j_1)$-colored $NES$ paths as $NES^*$ paths.
\end{remark}

\begin{proposition}\label{prop:dualwords}
    Let $w$ be an NES type 3 path with endpoint height $h$. The number of occurences of $w$ in NES type 1 paths from $(0,0)$ to $(n,0)$ at $y=j$ is equal to the number of occurences of $w^*$ in corresponding $NES^*$ type 1 paths at $y=j+h$.
\end{proposition}

\begin{proof}
    By Proposition~\ref{prop:dualbijection}, $p\mapsto p^*$ induces a bijection between NES type 1 paths and $NES^*$ type 1 paths from $(0,0)$ to $(n,0)$.

    Now, consider an arbitrary NES type 1 path from $(0,0)$ to $(n,0)$ with an occurence of $w$, which we call $pwq$. Then, $(pwq)^*=q^* w^* p^*$ has a corresponding occurence of $w^*$. Furthermore, we know that $w$ occurs at $y=j$, so it ends at height $j+h$. Then, $q$ goes from $y=j+h$ to $y=0$, so its endpoint height is $-(j+h)$. It follows that $q^*$ has an endpoint height of $j+h$ since we swap the relative numbers of $N$ and $S$ steps. As such, our corresponding occurence of $w^*$ in $q^* w^* p^*$ occurs at $y=j+h$.

    By an identical argument, since $w^*$ has an endpoint height of $-h$, occurences of $w^*$ at $y=j+h$ in $NES^*$ type 1 paths from $(0,0)$ to $(n,0)$ correspond to occurences of $w$ in NES type 1 paths at $y=j$. Since $((pwq)^*)^*=pwq$, this is sufficient to conclude that occurences of $w$ at $y=j$ in NES type 1 paths from $(0,0)$ to $(n,0)$ are in bijection with occurences of $w^*$ at $y=j+h$ in $NES^*$ type 1 paths from $(0,0)$ to $(n,0)$ as desired.
\end{proof}

\begin{lemma}\label{lemma:nes1wordcount}
    Let $w$ be an NES type 1 path from $(0,0)$ to $(\ell,h)$. The number of occurences of $w$ in NES type 1 paths from $(0,0)$ to $(n,0)$ at height $y=j$ is equal to:
    \begin{enumerate}
        \item[(i)] $\frac{1}{j_3}{(\frac{j_1}{j_3})}^j s_{n-\ell+1,1+h+2j}^*$ if $w$ does not end on $N$.
        \item[(ii)] $\frac{1}{j_3}{(\frac{j_1}{j_3})}^j (s_{n-\ell+1,1+h+2j}^* - j_3*s_{n-\ell,h+2j}^*)$ if $w$ ends on $N$. 
    \end{enumerate}
\end{lemma}

\begin{proof}
    By Proposition~\ref{prop:dualwords}, this is equivalent to counting $w^*$ in $NES^*$ type 1 paths from $(0,0)$ to $(n,0)$ at height $y=j+h$. We will first consider the case where $w$ does not end on $N$. This means that $w^*$ does not start on $S$.

    Let $pw^* q$ be an $NES^*$ type 1 path from $(0,0)$ to $(n,0)$ where $w^*$ occurs at $y=j+h$. Then, we  add an $N$ step and perform circular rotation to obtain $w^* qNp$. Finally, we remove $w^*$ to obtain $qNp$. Since $pw^* q$ is an $NES^*$ type 1 path, $p$ cannot start on $S$, so $qNp$ does not have any peaks. As such, we have constructed $j_3$ paths of $NES^*$ type 3, depending on which color we use for the $N$ that we added. We note that $qNp$ is a path from $(0,0)$ to $(n-\ell+1,1+h)$. Furthermore, since $w^*$ occurs at $y=j+h$, $j+h$ must be the ending height of $p$. Since $pw^* q$ ends at $y=0$, we further conclude that $q$ has an ending height of $-j$. Now, this means that $qNp$ reaches a height of $y=-j$. Since $pw^* q$ is an $NES^*$ type 1 path, $Np$ cannot go below its starting height of $y=-j$. Likewise, as $pw^* q$ ends at $y=0$, $q$ may not go below its ending height, which is $y=-j$. Thus, $qN p$ reaches $y=-j$ as its lowest height. In this way, we have exactly $j_3$ paths of $NES^*$ type 3 with lowest height of $y=-j$ from $(0,0)$ to $(n-\ell+1,1+h)$.

    Now, say we are given any path $x$ of $NES^*$ type 3 with a lowest height of $y=-j$ from $(0,0)$ to $(n-\ell+1,1+h)$. Then, we construct $w^* x$. This has no peaks, since $w$ is an NES type 1 path, so it cannot start on $S$, which means that $w^*$ cannot end on $N$. As such, $w^* x$ is an $NES^*$ type 3 path from $(0,0)$ to $(n+1,1)$. This means that there is exactly one more $N$ step than $S$ step, so by Proposition~\ref{prop:cyclelemma}, there exists a unique circular rotation of $w^* x$ to start on a step of type $N$ where $N$ and $S$ steps in that circular rotation form a dominating sequence. Furthermore, this step of type $N$ cannot occur in $w^*$, as this would mean that a right prefix of $w^*$ forms a dominating sequence. This is because it would then follow that the corresponding left prefix of $w$ must stay strictly below its starting point, which is impossible as $w$ is an NES type 1 path.
    
    So, call this circular rotation $Np w^* q$. We note that since $w$ does not end on $N$, $w^*$ does not start on $S$, so $Np w^* q$ is peakless. Since $Np w^* q$ forms a dominating sequence, $p w^* q$ must always stay at $y\ge 0$, so $p w^* q$ is an $NES^*$ type 1 path from $(0,0)$ to $(n,0)$. Furthermore, $x=qNp$ where $Np$ must never go below its starting height and $q$ must never go below its ending height since $pw^* q$ is an $NES^*$ type 1 path ending at $y=0$. As such, the ending height of $q$ must be the lowest height in $x$, which means that $q$ ends at $y=-j$. Since $x$ ends at $y=1+h$, this means that $p$ must end at $y=j+h$. Thus, $pw^* q$ is an $NES^*$ type 1 path from $(0,0)$ to $(n,0)$ where we have specified an occurence of $w^*$ at $y=j+h$.

    From these two processes, we see that each occurence of $w^*$ corresponds to exactly $j_3$ paths of $NES^*$ type 3 from $(0,0)$ to $(n-\ell+1,1+h)$ with a lowest point at $y=-j$. By Proposition~\ref{prop:lowestpointenumeration}, such paths are counted by ${(\frac{j_1}{j_3})}^j s_{n-\ell+1,1+h+2j}^*$. So, occurences of $w^*$ are counted by $\frac{1}{j_3}{(\frac{j_1}{j_3})}^j s_{n-\ell+1,1+h+2j}^*$ as desired.

    Next, we consider the case where $w$ ends on $N$. The argument is exactly as above, except that we require that $p$ does not end on $N$, since this would mean that $pw^* q$ has a peak. By an argument similar to Proposition~\ref{prop:gpathends}, $NES^*$ type 3 paths with a lowest height of $y=-j$ from $(0,0)$ to $(n-\ell,h)$ each correspond to exactly $j_3$ such paths from $(0,0)$ to $(n-\ell+1,h+1)$ not ending on $N$. Due to this, using the same processes as above,  $\frac{1}{j_3}{(\frac{j_1}{j_3})}^j (s_{n-\ell+1,1+h+2j}^* - j_3*s_{n-\ell,h+2j}^*)$ counts the occurences of $w^*$.
\end{proof}

\begin{corollary}
    Let $w$ be an NES type 1 path from $(0,0)$ to $(\ell,h)$. The number of occurences of $w$ in NES type 1 paths from $(0,0)$ to $(n,0)$ at height $y\ge j$ is equal to:
    \begin{enumerate}
        \item[(i)] $\frac{1}{j_3} {(\frac{j_1}{j_3})}^j g_{n-\ell+1,1+h+2j}^*$ if $w$ does not end on $N$.
        \item[(ii)] $\frac{1}{j_3}{(\frac{j_1}{j_3})}^j (g_{n-\ell+1,1+h+2j}^* - j_3*g_{n-\ell,h+2j}^*)$ if $w$ ends on $N$. 
    \end{enumerate}
\end{corollary}

\begin{proof}
    Use Lemma~\ref{lemma:nes1wordcount} and take sums of results for each height $\ge j$.
\end{proof}

\begin{theorem}\label{thm:countingwords}
    Let $w$ be an NES type 3 path from $(0,0)$ to $(\ell,h)$ with a lowest point of $y=-k$, where $h\ge 0$. The number of occurences of $w$ in NES type 1 paths from $(0,0)$ to $(n,0)$ at height $y=j\ge k$ is:
    \begin{enumerate}
        \item[(i)] $\frac{1}{j_3}{(\frac{j_1}{j_3})}^j s_{n-\ell+1,1+h+2j}^*$ if $w$ does not start on $S$ or end on $N$.
        \item[(ii)] $\frac{1}{j_3}{(\frac{j_1}{j_3})}^j (s_{n-\ell+1,1+h+2j}^* - j_3*s_{n-\ell,h+2j}^*)$ if $w$ does not start on $S$ but does end on $N$. 
        \item[(iii)] $\frac{1}{j_3}{(\frac{j_1}{j_3})}^j (s_{n-\ell+1,1+h+2j}^* - j_1*s_{n-\ell,h+2j}^*)$ if $w$ starts on $S$ but does not end on $N$.
        \item[(iv)]$\frac{1}{j_3}{(\frac{j_1}{j_3})}^j (s_{n-\ell+1,1+h+2j}^* - (j_1 + j_3)s_{n-\ell,h+2j}^* + j_1j_3*s_{n-\ell-1,h+2j-1})$ if $w$ starts on $S$ and ends on $N$.  
    \end{enumerate}
\end{theorem}

\begin{proof}
    First, consider the case where $w$ does not start on $S$ or end on $N$. It is easy to see that occurrences of $w$ in NES type 1 paths from $(0,0)$ to $(n,0)$ correspond bijectively to occurences of $wE$ in NES type 1 paths from $(0,0)$ to $(n+1,0)$ by replacing $w$ with $wE$ or vice versa. Now, $wE$ is a permutation of the word $x=N^{i_1} E^{i_2} S^{i_3}$. By replacing $wE$ with $x$ or vice versa, we again see that occurences correspond bijectively in NES type 1 paths from $(0,0)$ to $(n+1,0)$ at height $y=j$. This is because no new peaks may be added and since any occurence of $x$ counted starts at $y=j$, the lowest point after adding in $wE$ is $j-k\ge 0$. Since $h\ge 0$, we see that $x$ is an NES type 1 path. As such, by Lemma~\ref{lemma:nes1wordcount}, we may count its occurences at $y=j$ by $\frac{1}{j_3}{(\frac{j_1}{j_3})}^j s_{n-\ell+1,1+h+2j}^*$. 

    Next, consider the case where $w$ does not start on $S$ but does end on $N$. This again corresponds to occurences of $x$ except when $x$ is followed by $S$. If $h\ge 1$, then this is easy, as $xS$ for any color of $S$ is an NES type 1 path, so by Lemma~\ref{lemma:nes1wordcount}, its occurrences at $y=j$ are counted by $\frac{1}{j_3}{(\frac{j_1}{j_3})}^j j_3*s_{n-\ell,h+2j}^*$. This gives us $\frac{1}{j_3}{(\frac{j_1}{j_3})}^j (s_{n-\ell+1,1+h+2j}^* - j_3*s_{n-\ell,h+2j}^*)$ as desired. Meanwhile, if $h=0$, then we instead count $x^*=N^{i_3+1} E^{i_2} S^{i_1}$ in $NES^*$ paths from $(0,0)$ to $(n,0)$ at $y=j-1$. This gives us $\frac{j_3}{j_1} {(\frac{j_3}{j_1})}^{j-1} s_{n-\ell,2j}$ by Lemma~\ref{lemma:nes1wordcount}. It is easy to see that ${(\frac{j_3}{j_1})}^j s_{n-\ell,2j}={(\frac{j_1}{j_3})}^j s_{n-\ell,2j}^*$. So, we again obtain the formula $\frac{1}{j_3}{(\frac{j_1}{j_3})}^j (s_{n-\ell+1,1+h+2j}^* - j_3*s_{n-\ell,h+2j}^*)$ as desired.

    Now, consider the case where $w$ starts on $S$ but does not end on $N$. In this case, we want to count occurences of $x$ except when preceded by $N$. Since $x$ is an NES type 1 path, $Nx$ is as well, so by Lemma~\ref{lemma:nes1wordcount}, we count occurrences of $Nx$ at $y=j-1$ for each of the $j_1$ colors of $N$. This gives us our desired formula of $\frac{1}{j_3}{(\frac{j_1}{j_3})}^j (s_{n-\ell+1,1+h+2j}^* - j_1*s_{n-\ell,h+2j}^*)$.

    Finally, if $w$ starts on $S$ and ends on $N$, we count occurrences of $x$, then ignore occurences of $xS$ and $Nx$. However, this removes occurences of $NxS$ twice, so we must again add that value to obtain the correct formula. Using Lemma~\ref{lemma:nes1wordcount}, this gives us the formula above that we desired.
\end{proof}

\begin{corollary}\label{cor:ignoringheight}
    Let $w$ be an NES type 3 path from $(0,0)$ to $(\ell,h)$ with a lowest point of $y=-k$, where $h\ge 0$. The number of occurences of $w$ in NES type 1 paths from $(0,0)$ to $(n,0)$ at height $y\ge j\ge k$ is:
    \begin{enumerate}
        \item[(i)] $\frac{1}{j_3}{(\frac{j_1}{j_3})}^j g_{n-\ell+1,1+h+2j}^*$ if $w$ does not start on $S$ or end on $N$.
        \item[(ii)] $\frac{1}{j_3}{(\frac{j_1}{j_3})}^j (g_{n-\ell+1,1+h+2j}^* - j_3*g_{n-\ell,h+2j}^*)$ if $w$ does not start on $S$ but does end on $N$. 
        \item[(iii)] $\frac{1}{j_3}{(\frac{j_1}{j_3})}^j (g_{n-\ell+1,1+h+2j}^* - j_1*g_{n-\ell,h+2j}^*)$ if $w$ starts on $S$ but does not end on $N$.
        \item[(iv)]$\frac{1}{j_3}{(\frac{j_1}{j_3})}^j (g_{n-\ell+1,1+h+2j}^* - (j_1 + j_3)g_{n-\ell,h+2j}^* + j_1j_3*g_{n-\ell-1,h+2j-1})$ if $w$ starts on $S$ and ends on $N$.  
    \end{enumerate}
\end{corollary}

\begin{remark}
    Counting occurrences of $w$ at $y\ge k$ counts all occurrences of $w$ regardless of height.
\end{remark}

\begin{proposition}\label{prop:countingwordsidentity}
    $s_{n,k}^*=g_{n,k}^*-\frac{j_1}{j_3} g_{n,k+2}^*$.
\end{proposition}

\begin{proof}
    Let $w=N^{k+1} E$ with some colors. By Corollary~\ref{cor:ignoringheight}, the occurrences of $w$ in NES type 1 paths from $(0,0)$ to $(n+k+1,0)$ are enumerated by $\frac{1}{j_3}*g_{n,k+2}^*$.

    Now, take some NES type 1 path $pwq$ from $(0,0)$ to $(n+k+1,0)$. Add an $N$ step of some color to the beginning of the path to obtain $Npwq$ and then perform a circular rotation to obtain $wqNp$. Finally, remove $w$ to obtain $qNp$. Since we may choose $j_1$ different colors for our added $N$ step, this gives us $j_1$ distinct paths of the form $qNp$. In addition, we note that $qNp$ must be an NES type 3 path from $(0,0)$ to $(n,-k)$. Furthermore, we know that $wqNp$ is an NES type 3 path from $(0,0)$ to $(n+k+2,1)$, with a distinct circular rotation $Npwq$ which is a dominating sequence. By the Proposition~\ref{prop:cyclelemma}, $Npwq$ is the only circular rotation of $wqNp$ which may be a dominating sequence, so we know that $wqNp$ is not a dominating sequence. Equivalently, if $w=Nw'$, then $w' qNp$ is not an NES type 1 path. 

    Next, consider some arbitrary NES type 3 path $x$ from $(0,0)$ to $(n,-k)$. Suppose that $w' x$ is not an NES type 1 path. Then, $wx$ is an NES type 3 path from $(0,0)$ to $(n+k+2,1)$ which is not itself a dominating sequence. As such, by Proposition~\ref{prop:cyclelemma}, there exists some unique circular rotation $Npwq$ of $wx$ which is a dominating sequence. Then, $pwq$ is an NES type 1 path from $(0,0)$ to $(n+k+1,0)$ where we have chosen an occurence of $w$.

    It is easy to see that from the two processes outlined above, each occurence of $w$ in NES type 1 paths from $(0,0)$ to $(n+k+1,0)$ corresponds to exactly $j_1$ NES type 3 paths $x$ from $(0,0)$ to $(n,-k)$ where $w' x$ is not an NES type 1 path. By Proposition~\ref{prop:dualbijection}, there are $g^*_{n,k}$ NES type 3 paths from $(0,0)$ to $(n,-k)$. In addition, it is easy to see that NES type 3 paths $x$ from $(0,0)$ to $(n,-k)$ where $w' x$ is an NES type 1 path are in bijection with NES type 1 paths from $(0,0)$ to $(n+k+1,0)$ starting with $w'$. Using the bijection given by Proposition~\ref{prop:sdualbijection}, these paths are in bijection with $NES^*$ type 1 paths from $(0,0)$ to $(n+k+1,0)$ ending in $(w')^*$. By Proposition~\ref{prop:spathends}, there are $s^*_{n,k}$ such paths. As such, we conclude that there are $g^*_{n,k}-s^*_{n,k}$ NES type 3 paths $x$ from $(0,0)$ to $(n,-k)$ where $w' x$ is not an NES type 1 path. 

    Then, we conclude that there are exactly $\frac{1}{j_1}(g^*_{n,k}-s^*_{n,k})$ occurences of $w$ in NES type 1 paths from $(0,0)$ to $(n+k+1,0)$. This means that
    \[\frac{1}{j_3}*g_{n,k+2}^* = \frac{1}{j_1}(g^*_{n,k}-s^*_{n,k})\]

    And so
    \[s_{n,k}^* = g_{n,k}^* - \frac{j_1}{j_3}*g_{n,k+2}^*\]

    as desired.
\end{proof}

\begin{remark}
     This provides an alternate combinatorial proof for Corollary~\ref{cor:alternateexplicit}.
\end{remark}

\begin{proposition}\label{prop:stepidentity}
    For $n\ge 1$, $n*s_{n,0} = 2(\frac{j_3}{j_1} g_{n,2}-j_3*g_{n-1,1}) + \frac{j_2}{j_1}g_{n,1}$.
\end{proposition}

\begin{proof}
    By definition, $s_{n,0}$ is the number of NES type 1 paths from $(0,0)$ to $(n,0)$. As such, $n*s_{n,0}$ is the total number of steps in these NES type 1 paths. We may equivalently take the sum of the numbers of $N$, $E$, and $S$ steps in order to compute the same value.

    Consider $w=N$ for some color of $N$. Then, $w$ is an NES type 3 path from $(0,0)$ to $(1,1)$ with a lowest point at $y=0$. By Corollary~\ref{cor:ignoringheight}, the number of occurences of $w$ in NES type 1 paths from $(0,0)$ to $(n,0)$ is equal to $\frac{1}{j_3} (g^*_{n,2}-j_3*g^*_{n-1,1})$. We equivalently may express this as $\frac{j_3}{j_1^2} g_{n,2}-\frac{j_3}{j_1} g_{n-1,1}$. Then, we have $j_1$ colors of $N$ steps, so in total, we have $\frac{j_3}{j_1} g_{n,2}-j_3*g_{n-1,1}$ $N$ steps.

    Next, consider $w=E$ for some color of $E$. Then, $w$ is a path from $(0,0)$ to $(1,0)$ with a lowest point at $y=0$, so by Corollary~\ref{cor:ignoringheight}, the number of occurences of $w$ is $\frac{1}{j_3} g^*_{n,1}=\frac{1}{j_1}g_{n,1}$. Since we have $j_2$ colors for $E$ steps, this gives us $\frac{j_2}{j_1}g_{n,1}$ total steps.

    Finally, we know that in any $NES$ type 1 path from $(0,0)$ to $(n,0)$, there must be an equal number of $N$ and $S$ steps. As a result, there are $\frac{j_3}{j_1} g_{n,2}-j_3*g_{n-1,1}$ $S$ steps in total. We now conclude that
    \[n*s_{n,0} = 2(\frac{j_3}{j_1} g_{n,2}-j_3*g_{n-1,1}) + \frac{j_2}{j_1}g_{n,1}\]

    as desired.
\end{proof}

\begin{remark}
    This gives us a new formula for $s_{n,0}$ for all $n\ge 1$.
\end{remark}

\section{Applications}\label{sec:applications}
This section provides various explicit formulas which are easy to obtain using Theorem~\ref{thm:countingwords}. These formulas denote the number of occurences of various patterns in peakless Motzkin paths of length $n$ and correspond to sequences in~\cite{oeis}, denoted by the ascension number AXXXXXX. Furthermore, for the remainder of the paper, unless otherwise stated, we  will use $j_1=j_2=j_3=1$.

\begin{table}[h!]
\centering
\setlength\extrarowheight{1pt}
\begin{tabular}{|l|l|l|l|l|l|l|}
\hline
Pattern & Sequence & Formula & OEIS \\
\hline
$E$ & 1, 2, 4, 10, 24, \ldots & $g_{n,1}$ & A110236  \\
\hline
$NE$ & 1, 3, 7, 18, 46, \ldots & $g_{n-1,2}$ & A114713
 \\
\hline
$N$ & 1, 3, 8, 22, 58, \ldots & $g_{n,2}-g_{n-1,1}$ & A110239
 \\
\hline
$N+E$ & 1, 2, 5, 13, 32, \ldots & $g_{n+1,1}-g_{n,0}$ & A110320
 \\
\hline
$E$, $y=0$ & 1, 2, 3, 6, 13, \ldots & $s_{n,1}$ & A089735
 \\
\hline
$NE$, $y=0$ & 1, 3, 6, 13, 30, \ldots & $s_{n-1,2}$ & A098075
 \\
\hline
$N$, $y=0$ & 1, 3, 7, 17, 41, \ldots & $s_{n,2}-s_{n-1,1}$ & A089737
 \\
\hline
$SN$ & 1, 4, 13, 40, 116, \ldots & $g_{n-1,3}-2g_{n-2,2}+g_{n-3,1}$ & $\Delta$ of A190163
 \\
\hline
$NE^j N$ & 1, 5, 17, 52, 150, \ldots & $\sum_j (g_{n-j-1,3}-g_{n-j-2,2})$ & A187257
 \\
\hline
$NE^j S$ & 1, 3, 7, 17, 41, \ldots & $\sum_j g_{n-j-1,1}$ & A089742
 \\
\hline
$N+E+S$, $y=0$ & 1, 2, 5, 12, 27, \ldots & $2(s_{n,2}-s_{n-1,1})+s_{n,1}$ & A128096
 \\
\hline$N+E+S$ & 1, 2, 6, 16, 40, \ldots & $2(g_{n,2}-g_{n-1,1})+g_{n,1}$ & $n*A004148$
 \\
\hline
\end{tabular}
\end{table}

Note that any formula counting occurrences of $w$ likewise counts occurences of $w^*$. For instance, the formula for $N$ likewise counts occurences of $S$. We further note that taking the formulas above to compute $N+E+S$, we obtain a new formula for $n*s_{n,0}$, which is given for $j_1=j_2=j_3=1$ above and for arbitrary $j_1,j_2,j_3\ge 1$ in Proposition~\ref{prop:stepidentity}. Finally, we note that using Theorem~\ref{thm:countingwords}, we find that $s_{n,1+2j}$ counts peakless Motzkin paths of length $n$ with exactly 1 marked $E$ step at height $y=j$. This interestingly corresponds to peakless Motzkin paths with $1+2j$ marked $E$ steps at $y=0$, using an interpretation of RNA Array I explored in~\cite{markedlevelsteps}.

\subsection{Order Ideals and Whitney Numbers}
Next, we investigate certain connections between NES type 3 paths with $j_1=j_2=j_3=1$ and the Whitney numbers of certain distributive lattices by defining a partial order on NES type 3 paths ending at $x=n$. We will see that this partial order forms a lattice isomorphic to the lattice of order ideals of the fence of order $2n$ which is defined and investigated in~\cite{whitneylattices}. This investigation is motivated by the fact that the sequence $(g_{n,0})_{n\ge 0}$ is denoted A051286 in~\cite{oeis} and represents Whitney numbers of the lattice of the ideals of the fence of order $2n$.

\begin{definition}
    We define a partial ordering on NES paths of type 3 ending at $x=n$ as follows. Order steps by $S < E < N$. Then, taking $p$ and $q$ in their linear forms, $p \le q$ iff the $i$th step of $p$ is always less than or equal to the $i$th step of $q$.
\end{definition}

\begin{proposition}
    NES paths of type 3 ending at $x=n$ form a lattice.
\end{proposition}

\begin{proof}
    We have a partial ordering on NES type 3 paths ending at $x=n$, so it is sufficient to show that given two paths $p$ and $q$, $p\lor q$ and $p\land q$ are well defined.

    Consider $p$ and $q$ in their linear forms. Order steps by $S < E < N$. Then, generate $p\lor q$ by taking both $p$ and $q$ in their linear forms. We define the $i$th step in the linear form of $p\lor q$ as the greatest of the $i$th steps of $p$ and $q$. It is easy to see that $p,q \le p\lor q$. Furthermore, if $p,q \le r$, then the $i$th step of $r$ is greater than or equal to both the $i$th steps of $p$ and $q$ by definition. By definition, it then follows that $p\lor q\le r$. As such, we have properly defined $p\lor q$.

    Likewise, we define $p\land q$ by taking the smaller of the $i$th steps of $p$ and $q$. An entirely analogous argument holds. Since $p\lor q$ and $p\land q$ are always well-defined, we have a lattice.
\end{proof}

\begin{definition}
    We define the fence of order $2n$ as a partial ordering of $\{1,\ldots,2n\}$ where for all $1\le i,j\le 2n$, if $i$ is even and $j=i\pm 1$, $i < j$.
\end{definition}

\begin{definition}
    We define the crown of order $2n$ as a partial ordering of $\{1,\ldots,2n\}$ where for all even $i$, $i\pm 1 < i$, where we take $i\pm 1$ mod $2n$.
\end{definition}

\begin{definition}
    Let $P$ be a partially ordered set. Let $I\subset P$. If for all $y\in I$ and $z\in P$ such that $z\le y$, $z\in I$, we call $I$ an order ideal of $P$.
\end{definition}

\begin{remark}
    The fence and crown and their order ideals are investigated in~\cite{whitneylattices}.
\end{remark}

\begin{example}
    $\{1 < 2 > 3 < 4 > 5 < 6\}$ is the fence of order 6.
\end{example}

\begin{example}
    $\{1 < 2 > 3 < 4 > 5 < 6 > 1\}$ is the crown of order 6.
\end{example}

\begin{proposition}
    NES paths of type 3 from $(0,0)$ to $(n,k)$ correspond to order ideals of size $n+k$ for the fence of order $2n$.
\end{proposition}

\begin{proof}
    Consider some NES type 3 path from $(0,0)$ to $(n,k)$ in its linear form. Then, map every $S$ step to $00$, every $E$ step to $10$, and every $N$ step to $11$. Then, we have a string of 0s and 1s of length $2n$. We define an order ideal $I$ of the fence of order $2n$ so that $i\in I$ iff the $i$th entry of this string is equal to 1. If $i$ is even, then since 00, 10, and 11 are the only options, this satisfies our requirement from $i-1 < i$ for the order ideal. Likewise, as $NS$ is impossible, $i+1 < i$ is satisfied. Finally, we note that each $S$ step contributes 2 elements not in $I$, each $E$ step contributes 1 in $I$ and one not, and each $N$ step contributes 2 in $I$. As such, $|I|-(2n-|I|)=2(|I|-n)$ counts double the difference between $N$ and $S$ steps, which we know is $k$. As $2(|I|-n)=2k$, $|I|=n+k$. Thus, we have constructed an order ideal of size $n+k$. It is easy to see that our process is entirely reversible, so this is a bijection.
\end{proof}

\begin{remark}
    So, $g_{n,k}$ counts the number of order ideals of size $n+k$ for the fence of order $2n$. Note that this bijection also preserves the lattice structure and so defines an isomorphism of lattices.
\end{remark}

\begin{example}
    Take the path $NESN$. This gives us the string $11100011$. From this, we obtain the order ideal $I=\{1,2,3,7,8\}$ for the fence of order 8.
\end{example}

\begin{proposition}
    There are $g_{n,k}-g_{n-2,k}$ order ideals of size $n+k$ for the crown of order $2n$.
\end{proposition}

\begin{proof}
    Let $I$ be an order ideal of order $n+k$ for the crown of order $2n$. The crown is a refinement of the fence, so $I$ is also an order ideal for the fence, which is counted by $g_{n,k}$. However, the crown also requires that $2n < 1$, which is equivalent to requiring that the corresponding path does not end on $N$ and start on $S$. It is easy to show that such paths are counted by $g_{n-2,k}$ by removing the $S$ from the start and the $N$ from the end. As such, these order ideals are counted by $g_{n,k}-g_{n-2,k}$.
\end{proof}

\begin{remark}
    This appears relevant for a conjecture given for A110320 in~\cite{oeis}.
\end{remark}

\begin{proposition}
    The order ideals of order $n+k$ for the fence of order $2n+1$ are counted by $g_{n+1,k-1}-g_{n,k-2}$.
\end{proposition}

\begin{proof}
    Order ideals of order $n+k$ for the fence of order $2n+1$ correspond exactly to order ideals of order $(n+1)+(k-1)$ for the fence of order $2(n+1)$ which do not contain $2n+2$. These correspond to NES type 3 paths from $(0,0)$ to $(n+1,k-1)$ which do not end in $N$. Such paths are counted in general by $g_{n+1,k-1}$, and if they end in $N$, they correspond to paths counted by $g_{n,k-2}$ via Proposition~\ref{prop:gpathends}. As such, such paths not ending in $N$ are counted by $g_{n+1,k-1}-g_{n,k-2}$.
\end{proof}

\begin{remark}
    We now have explicit formulas for A051286, A051292, and A051291 in~\cite{oeis} based on entries of RNA Array III.
\end{remark}

\subsection{Asymptotics of Patterns}
Given the formulas to compute patterns in terms of $s_{n,k}$ and $g_{n,k}$ which we have derived, it is essentially sufficient to know certain asymptotics of $s_{n,k}$ and $g_{n,k}$ in order to compute asymptotics for any given pattern. We compute the formulas relevant to the asymptotics of pattern statistics. It should be noted that a number of related asymptotics are given in the context of RNA secondary structure enumeration by~\cite{combinatoricssecondarystructures}. In this section, we will use arbitrary $j_1,j_2,j_3\ge 1$. 

\begin{proposition}\label{prop:asymptoticsform}
    Given $j_1,j_2,j_3\ge 1$,
    \[s(z) = \frac{p_1(z)}{z^2} - \frac{\sqrt{p_2(z)}}{z^2}(1-\frac{z}{\alpha})^{1/2}\]

    for polynomials $p_1(z), p_2(z)$, $\alpha=\frac{(j_2+2\sqrt{j_1 j_3}) - \sqrt{(j_2+2\sqrt{j_1 j_3})^2 - 4j_1 j_3}}{2j_1 j_3}$ and $\beta=\frac{p_1(\alpha)}{\alpha^2}=\frac{(j_1 j_3)^{-1/2}}{\alpha}$. Furthermore, $s(z)$ satisfies the requirements for Theorem 4.1 in~\cite{combinatoricssecondarystructures}.
\end{proposition}

\begin{proof}
    Using the identity in Proposition~\ref{prop:generatingexpression}, we know that
   \[s(z) = \frac{1-j_2 z + j_1 j_3 z^2 \pm \sqrt{(1-j_2 z+j_1j_3 z^2)^2 - 4j_1 j_3 z^2}}{2j_1 j_3 z^2}.\]

    Furthermore, we know that $s(z)$ is defined as a formal power series in $z$ without any negative powers of $z$. As such, $1-j_2 z + j_1 j_3 z^2 \pm \sqrt{(1-j_2 z+j_1j_3 z^2)^2 - 4j_1 j_3 z^2}$ must be of order 2. We know that $[z^0]\sqrt{(1-j_2 z+j_1j_3 z^2)^2 - 4j_1 j_3 z^2}=1$, so it follows that
    \[s(z) = \frac{1-j_2 z + j_1 j_3 z^2 - \sqrt{(1-j_2 z+j_1j_3 z^2)^2 - 4j_1 j_3 z^2}}{2j_1 j_3 z^2}.\] 

    Now, since $s(z)$ is the quotient of two analytic functions with zeroes of multiplicity 2 at $z=0$, it follows that $s(z)$ is analytic at $z=0$. In addition, $(1-j_2 z+j_1j_3 z^2)^2 - 4j_1 j_3 z^2 = 1 \ge 0$ at $z=0$. Furthermore, it is easy to show that the solution of $(1-j_2 z+j_1j_3 z^2)^2 - 4j_1 j_3 z^2=0$ with the smallest magnitude is $\alpha$. From this, it follows that on $|z| < \alpha$, $s(z)$ defined in this way is a composition of analytic functions, since $\sqrt{(1-j_2 z+j_1j_3 z^2)^2 - 4j_1 j_3 z^2}$ is analytic whenever $(1-j_2 z+j_1j_3 z^2)^2 - 4j_1 j_3 z^2 > 0$. Additionally, all terms of the formal power series for $s(z)$ are strictly positive, which is easy to see from the combinatorial interpretation of $s(z)$. As such, the positive real solution $z=\alpha$ is the only singularity on the circle of convergence for $s(z)$.
    
    Next, since $\alpha$ is a root for $(1-j_2 z+j_1j_3 z^2)^2 - 4j_1 j_3 z^2$, we may factor out a term of $(1-\frac{z}{\alpha})$ in order to obtain a new polynomial. This means that
    \[s(z) = \frac{p_1(z)}{z^2} - \frac{\sqrt{p_2(z)}}{z^2} {(1-\frac{z}{\alpha})}^{1/2}\]

    for polynomials $p_1,p_2$ as desired. This form and the convergence properties we have already obtained are sufficient to show that $s(z)$ satisfies the requirements given in~\cite{combinatoricssecondarystructures}. 
    
    Furthermore, $\frac{p_1(\alpha)}{\alpha^2}=\frac{1-j_2 \alpha + j_1 j_3 \alpha^2}{2j_1j_3 * \alpha^2}=\frac{2\sqrt{j_1 j_3} \alpha}{2j_1j_3 \alpha^2}=\frac{{(j_1 j_3)}^{-1/2}}{\alpha}$ as desired, using our definition of $\alpha$.
\end{proof}

\begin{remark}
    Since $s_{n,0} > 0$ for all $n\ge 0$, $s(z)$ has only positive coefficients. Due to this, along with the form for $s(z)$ given by Proposition~\ref{prop:asymptoticsform}, $s(z)$ satisfies all of the requirements for $y(z)$ given by Theorem 4.1, Corollary 4.2, and Corollary 4.3 in~\cite{combinatoricssecondarystructures}. 
\end{remark}

\begin{lemma}\label{lemma:gasymptotic}
    Consider $h(z)=\frac{1}{\alpha^2 \beta^2 - z^2 y^2} \Phi(z, y(z))$ where $y(z)$ fulfills the requirements of Corollary 4.3 in~\cite{combinatoricssecondarystructures} and $\Phi(z,y)$ is polynomial in $y$ and analytic in $z$. Then, $\lim_{n\to\infty} \frac{h_n}{n*y_n} = \frac{\Phi(\alpha,\beta)}{\alpha^2 \beta g^2(\alpha)}$.
\end{lemma}

\begin{proof}
    We know that 
    \[y(z) = \beta(z) + {(1-\frac{z}{\alpha})}^{1/2} g(z).\]

    Then, 
    \[y^2(z) = \beta^2(z) + (1-\frac{z}{\alpha})g^2(z) + 2(1-\frac{z}{\alpha})^{1/2} \beta(z)g(z).\]

    Next,
    \[\alpha^2 \beta^2 - z^2 y^2(z) = (1-\frac{z}{\alpha}) \phi(z) - 2(1-\frac{z}{\alpha})^{1/2} z^2 \beta(z)g(z)\]

    where $\phi$ is analytic near $\alpha$. Then, it follows that
    \[\frac{1}{\alpha^2 \beta^2 - z^2 y^2} = \frac{\phi(z)}{(1-\frac{z}{\alpha}) \phi^2(z) - 4 z^4 \beta^2(z)g^2(z)} + \frac{2z^2 \beta(z)g(z)}{(1-\frac{z}{\alpha}) \phi^2(z) - 4 z^4 \beta^2(z)g^2(z)}(1-\frac{z}{\alpha})^{-1/2}.\]

    We note that $\alpha^2 \beta^2 - z^2 y^2$ has a zero at $z=\alpha$, so $\frac{1}{\alpha^2 \beta^2 - z^2 y^2}$ has a singularity at $z=\alpha$. Then, by the requirements of Corollary 4.3 in~\cite{combinatoricssecondarystructures}, we know that all $y_n$ are nonnegative, and after some sufficiently large $N$, all $y_n$ are positive. It follows that the same holds for the coefficients of $y^2$. So, for all $|z| \le \alpha$, 
    \[|z^2 y^2(z)|\le |z|^2 y^2(|z|) \le \alpha^2 \beta^2\]
    
    with equality only at $z=\alpha$. As such, $\frac{1}{\alpha^2 \beta^2 - z^2 y^2}$ is analytic for $|z| < \alpha$ and has a singularity for some $z$ where $|z|=\alpha$ only at $z=\alpha$. This fulfills the requirements of Theorem 4.1 in~\cite{combinatoricssecondarystructures}, and it is easy to see that these requirements are still fulfilled if we take $h(z)=\frac{1}{\alpha^2 \beta^2 - z^2 y^2} \Phi(z,y(z))$. Then, Theorem 4.1 in~\cite{combinatoricssecondarystructures} for $h(z)$ gives us that 
    \[h_n \sim \frac{\Phi(\alpha, \beta)}{-2\alpha^2 \beta g(\alpha)} * \frac{1}{\Gamma(\frac{1}{2})} n^{-1/2} {(\frac{1}{\alpha})}^n = \frac{\Phi(\alpha,\beta)}{\alpha^2 \beta g(\alpha) \Gamma(-\frac{1}{2})} *n^{-1/2} {(\frac{1}{\alpha})}^n.\] 
    
    Meanwhile, we find that 
    \[y_n \sim \frac{g(\alpha)}{\Gamma(-\frac{1}{2})} n^{-3/2} {(\frac{1}{\alpha})}^{n}\]

    And so $\lim_{n\to\infty} \frac{h_n}{n*y_n} = \frac{\Phi(\alpha,\beta)}{\alpha^2 \beta g^2(\alpha)}$ as desired.
\end{proof}

\begin{remark}
    This uses the same argument structure as Corollary 4.3 in~\cite{combinatoricssecondarystructures}. We might describe $h(z)=\frac{1}{\alpha^m \beta^m - z^m y^m} \Phi(z,y)$. Then, Corollary 4.3 in~\cite{combinatoricssecondarystructures} proves a result for $m=1$ and Lemma~\ref{lemma:gasymptotic} proves a result for $m=2$. This could likely be generalized to all $m$, though doing so is beyond the scope of this paper. 
\end{remark}

\begin{theorem}\label{theorem:asymptotics}
    For any $j_1,j_2,j_3\ge 1$, $s_{n,k}$ and $g_{n,k}$ obey the following asymptotics:
    \begin{enumerate}
        \item[(i)] For fixed $k$, $\lim_{n\to\infty} s_{n,k}/s_{n,0}=(k+1){(\sqrt{\frac{j_1}{j_3}})}^k$.
        \item[(ii)] For fixed $k$, $\lim_{n\to\infty} g_{n,k}/g_{n,0}={(\sqrt{\frac{j_1}{j_3}})}^k$.
        \item[(iii)] For fixed $k$, $\lim_{n\to\infty} s_{n,k}/s_{n-1,k} = \lim_{n\to\infty} g_{n,k}/g_{n-1,k}=\frac{1}{\alpha}$ where $\alpha$ is given by Proposition~\ref{prop:asymptoticsform}.
        \item[(iv)] $\lim_{n\to\infty} \frac{n*s_{n,0}}{g_{n,0}}=2(1-\alpha\sqrt{j_1 j_3})+\frac{j_2}{\sqrt{j_1 j_3}}$.
    \end{enumerate}
\end{theorem}

\begin{proof}
    We know that $s_{n,k}$ has a generating function of $j_1^k z^k s^{k+1}(z)$. By Corollary 4.2 in~\cite{combinatoricssecondarystructures}, 
   \[\lim_{n\to\infty} \frac{s_{n,k}}{s_{n,0}} = j_1^k * (k+1)\alpha^k \beta^k = (k+1) j_1^k {(j_1 j_3)}^{-k/2}=(k+1){(\sqrt{\frac{j_1}{j_3}})}^k\]

    which yields the desired result for (i). Next, $g_{n,k}$ is generated by $\frac{j_1^k}{1-j_1 j_3 * z^2 s^2(z)} z^k s^{k+1}(z)$. We note that $j_1 j_3 = \frac{1}{\alpha^2 \beta^2}$, so $g_{n,k}=j_1^k \alpha^2 \beta^2 * \frac{1}{\alpha^2 \beta^2 - z^2 s^2(z)} * z^k s^{k+1}(z)$. We may now apply Lemma~\ref{lemma:gasymptotic} twice to conclude that $\lim_{n\to\infty} g_{n,k}/g_{n,0}=\frac{j_1^k \alpha^k \beta^{k+1}}{\beta}={(\sqrt{\frac{j_1}{j_3}})}^{k}$ since $\alpha\beta=\frac{1}{\sqrt{j_1 j_3}}$. This yields (ii).

    Now, for fixed $k$, $s_{n,k}$ is generated by $j_1^k z^k s^{k+1}(z)$, while $s_{n-1,k}$ is generated by $j_1^k z^{k+1} s^{k+1}(z)$. Applying Corollary 4.2 from~\cite{combinatoricssecondarystructures} twice, we find that
    \[\lim_{n\to\infty} \frac{s_{n,k}}{s_{n-1,k}}=\lim_{n\to\infty} \frac{s_{n,k}}{s_{n,0}}\frac{s_{n,0}}{s_{n-1,k}} = \frac{j_1^k * (k+1)\alpha^{k} \beta^k}{j_1^k * (k+1)\alpha^{k+1} \beta^k}=\frac{1}{\alpha}\]

    as desired. Likewise, we apply Lemma~\ref{lemma:gasymptotic} twice to conclude that
    \[\lim_{n\to\infty} \frac{g_{n,k}}{g_{n-1,k}}=\frac{\alpha^k \beta^{k+1}}{\alpha^{k+1}\beta^{k+1}} = \frac{1}{\alpha}.\]

    In this way, we have obtained (iii). Finally, by Proposition~\ref{prop:stepidentity},
    \[n*s_{n,0} = 2(\frac{j_3}{j_1} g_{n,2}-j_3*g_{n-1,1}) + \frac{j_2}{j_1}g_{n,1}.\]

    Using this identity,
    \[\lim_{n\to\infty} \frac{n*s_{n,0}}{g_{n,0}}=2(\frac{j_3}{j_1}\lim_{n\to\infty} \frac{g_{n,2}}{g_{n,0}}-j_3\lim_{n\to\infty} \frac{g_{n-1,1}}{g_{n,0}})+\frac{j_2}{j_1}\lim_{n\to\infty} \frac{g_{n,1}}{g_{n,0}}.\]

    Then, using the asymptotics we have already computed, we find that
    \[\lim_{n\to\infty} \frac{n*s_{n,0}}{g_{n,0}}=2(1-\alpha\sqrt{j_1 j_3})+\frac{j_2}{\sqrt{j_1 j_3}}\]
    
    as desired.
\end{proof}

\begin{example}
Using the above results for $j_1=j_2=j_3=1$, we obtain $\alpha=\frac{3-\sqrt{5}}{2}$ and $\beta=\frac{1}{\alpha}$. We then may compute asymptotics for various formulas. For instance, taking the formula $s_{n,2}-s_{n-1,1}$ which counts $N$ steps at $y=0$, we find that $\lim_{n\to\infty} \frac{s_{n,2}-s_{n-1,1}}{s_{n,0}}=3-2\alpha=\sqrt{5}$. Likewise, considering the formula $s_{n,1}$ for $E$ steps at $y=0$, we find that $\lim_{n\to\infty} \frac{s_{n,1}}{s_{n,0}}=2$. These results align with asymptotics given for A089737 and A089735 in~\cite{oeis}.
\end{example}

\begin{example}
    If we take $j_1=j_3=1$ and $j_2=2$, we obtain $\alpha=2-\sqrt{3}$ and $\beta=\frac{1}{\alpha}$. We then find that $\lim_{n\to\infty} \frac{n*s_{n,0}}{g_{n,0}}=2(1-\alpha)+2=2\sqrt{3}$. We note that both of these sequences, $s_{n,0}$ and $g_{n,0}$, appear in~\cite{oeis} as A187256 and A101500, so this gives an asymptotic relationship between two more known sequences.
\end{example}

\section{Further Work}\label{sec:conclusion}
\begin{openproblem}
We know that $S_{1,1,j_3}^{j_2}=S_{1,j_2,j_3}$ because $M\mapsto M^j$ has the effect of taking a $\Delta$-sequence generated by $B(z)$ and yielding a $\Delta$-sequence generated by $jB(z)$. A problem of interest is to find an operation on Riordan matrices which, given a $\Delta$-sequence generated by $B(z)$, yields one generated by $B(jz)$. This would allow us to find a relationship between $S_{1,1,1}$ and $S_{1,1,j_3}$.
\end{openproblem}

\begin{openproblem}
We have provided a method to obtain pattern statistics for Motzkin paths avoiding $NS$. These methods might be extended to Motzkin paths avoiding other patterns. For instance, pattern statistics for valleyless Motzkin paths might also be found, though upon preliminary inspection, that problem seems more difficult using this paper's methods.
\end{openproblem}

\begin{openproblem}
    We have provided a method to induce a distributive lattice on prefixes of grand peakless Motzkin paths. This method might be further generalized to other classes of lattice walks.
\end{openproblem}

\section*{Acknowledgements}
This work was completed when the author was visiting Yale University for the SUMRY REU. The author would also like to thank Swarthmore College for funding this work through the Allen and Naomi Schneider Summer Research Fund. Finally, the author would like to thank Professor Asamoah Nkwanta for discussions surrounding this project and helpful feedback on the paper.  

\bibliographystyle{IEEEtran}
\bibliography{walks_patterns_riordan}

\end{document}